\documentclass[11pt]{article}
\usepackage[numbers,square]{natbib}
\usepackage{amsmath}
\usepackage{amsthm}
\usepackage{amsfonts}
\usepackage{amssymb}
\usepackage{siunitx}
\usepackage{commath}
\usepackage{graphicx}
\usepackage{xspace}

\usepackage{color}
\usepackage{bm}
\usepackage[lofdepth,lotdepth]{subfig}
\usepackage{stmaryrd}
\usepackage{url}
\usepackage{amsthm}
\usepackage{footnote}
\usepackage{multirow} 
\makesavenoteenv{tabular}
\makesavenoteenv{table}
\usepackage[hidelinks]{hyperref}
\hypersetup{colorlinks = true, urlcolor = blue,
  linkcolor = blue, citecolor = blue}
\usepackage{cleveref}
\usepackage{xcolor}

\usepackage[mathscr]{euscript}
\usepackage[shortlabels]{enumitem}

\usepackage[margin=0.7in]{geometry}
\makeatletter
\newcommand{\tnorm}{\@ifstar\@tnorms\@tnorm}
\newcommand{\@tnorms}[1]{%
  \left|\mkern-1.5mu\left|\mkern-1.5mu\left|
   #1
  \right|\mkern-1.5mu\right|\mkern-1.5mu\right|
}
\newcommand{\@tnorm}[2][]{%
  \mathopen{#1|\mkern-1.5mu#1|\mkern-1.5mu#1|}
  #2
  \mathclose{#1|\mkern-1.5mu#1|\mkern-1.5mu#1|}
}
\makeatother

\newtheorem{theorem}{Theorem}
\newtheorem{lemma}{Lemma}
\newtheorem{corollary}{Corollary}

\title{An unfitted HDG method for a distributed optimal convection-diffusion control problem}
\author{ Esteban Henr\'iquez\thanks{Department of Applied Mathematics,
    University of Waterloo, ON, Canada (\url{ehenriqu@uwaterloo.ca}),
    \url{https://orcid.org/0000-0002-0243-0368}}
  \and
  Manuel Solano\thanks{Departamento de Ingenier\'ia Matem\'atica, Facultad de Ciencias F\'isicas y Matem\'aticas, Universidad de Concepci\'on, Concepci\'on, Chile. (\url{msolano@ing-mat.udec.cl}),
    \url{https://orcid.org/0000-0001-8589-3685}}
  }
\begin{document}
\maketitle
\begin{abstract}
We analyze a high order unfitted hybridizable discontinuous Galerkin (HDG) method for an optimal control problem governed by a convection-diffusion equation posed in a domain with piecewise-wise $\mathcal{C}^2$ boundary $\partial \Omega$. The computational domain $\Omega_h$ does not necessarily fit $\Omega$ and the Transfer Path Method (TPM) is used to transfer the boundary data from  $\partial \Omega$ to $\partial \Omega_h$ through segments of direction $\boldsymbol{m}$. Under closeness conditions between $\partial \Omega_h$ and $\partial \Omega$ and on the transfer vector $\bm m$, we prove optimal order of convergence in the $L^2$-norm for all variables of the state and adjoint problems. We also show numerical examples to complement the theory.
\end{abstract}

\section{Introduction}\label{ch:introduction}

Numerous applications in engineering, physics, medicine and other fields are analyzed within the framework of optimal control of distributed parameter systems \cite{glowinski1995exact,lions1971optimal,troltzsch2024optimal,zuazua2007controllability}. In these systems, mathematical formulations often employ boundary controls or internally locally distributed controls, which are known as additive controls because of their appearance as additive terms in the model equations. On the other hand, many physical problems of interest are modelled by linear convection-diffusion partial differential equations, such as predicting the distribution of a substance in a continuous medium, for instance, pollutants in air or water. In this context, it is beneficial to manipulate the source terms (e.g., pollutant emission rates) to ensure that the solution of the PDE conforms to a desired distribution or keeps the pollutant concentration below a specified threshold. This problem can be approached using optimal control theory, where the source term acts as a control function, and the observation is a function of the solution of the PDE. We can find different numerical methods in the literature to solve these type of problems, such as finite differences \cite{mctaggart2004finite}, standard finite elements \cite{quarteroni2006numerical}, finite volumes \cite{bochev2015multiscale}, discontinuous Galerkin (DG) \cite{yucel2015adaptive,yucel2012distributed, zhou2014local}, and hybridizable discontinuous Galerkin (HDG) \cite{chen2018hdg, chen2019hdg, hu2018superconvergent}. The latter is the approach of this work.

Discontinuous Galerkin methods are known for their flexibility to handle complex meshes, their high accuracy, and their ability to deal with discontinuities in the solution. In addition, they allow efficient local refinement, conserve important physical quantities, and handle complex boundary conditions well. They are also noted for their ease of parallelization, which makes them suitable for simulations in clusters. The computational efficiency can be further enhanced by a hybridization process. The HDG method introduced by \cite{cockburn2009unified} considers auxiliary variables at the interfaces of the elements, which implies a reduction in the size of the global linear system by using a static condensation procedure. This type of HDG method has been applied to a variety of different problems \cite{cesmelioglu2013analysis, 
cesmelioglu2017analysis,cockburn2011analysis,fu2015analysis,rhebergen2013space,
sanchez2019hybridizable}.

Several existing numerical methods for solving PDEs have usually been developed assuming polygonal/polyhedric domains. However, in real scenarios, the PDEs are usually posed on curved domains. In this direction, two possible approaches can adopted: \textit{fitted} or \textit{unfitted} methods. The former, adjust the mesh to the real domain using an appropriate interpolation of the boundary \cite{lenoir1986optimal, li2010optimal}. However, these methods tend to be inadequate because of the high computational cost for certain problems, such as in the case of shape optimization or moving domains, where remeshing would be necessary. On the other hand, \textit{unfitted} schemes have the flexibility of working with non-body fitted grids, but it is not straightforward to obtain a high-order method since the true boundary data is posed ``away" from the computational boundary. This produces a geometric error that usually dominates the finite element error.

During the last decade, the community has been interested in the development of high-order \textit{unfitted} schemes. Consequently, several methods have been proposed: the Cut Finite Element Method \cite{burman2015cutfem, burman2022cutfem, frachon2019cut}, the Shifted Boundary Method \cite{atallah2022high, main2018shifted}, and the Transfer Path Method (TPM). The latter is the focus of our work. The TPM was originally developed in \cite{cockburn2012solving,paper2}, although the name {\it  Transfer Path Method} was first used later in \cite{SaSaSo2022b}. It has been applied to different problems, mostly on HDG discretization: Stokes flow \cite{solano2019high}, Oseen equations \cite{solano2022unfitted}, linear elasticity equations \cite{cardenas2024high}, Helmholtz equation \cite{camargo2021high}, convection diffusion equations \cite{cockburn2014solving}, Grad-Shafranov equations \cite{sanchez2020adaptive}, and recently to a distributed control problem \cite{henriquez2023control}, interface problems using non-matching grids \cite{bermudez2026hybridizable,MaNgSo2022}, and shape optimization \cite{henriquez2025unfitted}. In addition, the TPM has also been developed for conforming mixed finite element methods \cite{OyZuSo2019,OySoZu2020}.

This manuscript is closely related to our previous work in \cite{henriquez2023control} and some of the estimates that we will present are a generalization of what we obtained in that paper. However, we would like to highlight new aspects that the current manuscript will address. First, although the TPM was numerically developed for convection-diffusion in the context of HDG methods \cite{cockburn2014solving}, to date, a comprehensive theoretical convergence analysis is still lacking. In this work, we fill that gap and present a thorough convergence analysis of the convection-diffusion problem for that discretization and extend it to a distributed control problem. Secondly, in all the aforementioned works on the analysis of the TPM, the authors assumed, for simplicity, that the tangent vector $\bm m$ of the transfer path is parallel to the normal vector $\bm n$ of the boundary facets. They also claimed that all estimates hold if that assumption is not true, as long as $\bm m$ does not deviate too much from $\bm n$.
In this paper, we prove this claim and quantify how far $\bm m$ can be from $\bm n$ to guaranty convergence of the scheme.

We combine the work in \cite{chen2018hdg} and \cite{paper2}, in order to propose and analyze an HDG method for an optimal control convection-diffusion problem posed in a non-polygonal/polyhedral domain using the TPM. To state the problem, let $\Omega$ be a Lipschitz domain in $\mathbb{R}^n$, $n \in\{2,3\}$, with a piecewise $\mathcal{C}^2$ boundary denoted by $\Gamma := \partial\,\Omega$. Given a source term $f \in L^2(\Omega)$, a target function $y_d \in L^2(\Omega)$, and $g\in H^{1/2}(\Gamma)$, our objective is to minimize the functional
\begin{equation}
   J(u) := \frac{1}{2} \norm[0]{y - y_d}_{L^2(\Omega)}^2
    + \frac{\gamma}{2} \norm[0]{u}_{L^2(\Omega)}^2
\end{equation}
over  $u \in L^2(\Omega)$, subject to
\begin{subequations}\label{eq:state_eq}
\begin{alignat}{2}
    \label{eq:state_eq_a}
    - \Delta y + \bm\beta \cdot \nabla y 
    &  = f +u \qquad &&\text{in } \Omega, \\
    \label{eq:state_eq_b}
    y & = g \qquad && \text{on }\Gamma,
\end{alignat}
\end{subequations}
where  $\gamma\,>\,0$ is a given regularization parameter.  and we assume that the vector field $\bm\beta\in [W^{1,\infty}(\Omega)]^n$ satisfies $\nabla \cdot \bm\beta = 0$.

To characterize the variable $u$, we employ the adjoint state equation. Specifically, the state $y$ and control $u$ variables are solutions to the optimal control problem (\ref{eq:state_eq}) if and only if they satisfy the subsequent optimality system.

\begin{subequations}\label{eq:optimiality-system}
    \begin{minipage}{0.45\textwidth}
\begin{alignat}{2}
\label{eq:optimiality-system-a}
    - \Delta y + \bm\beta \cdot \nabla y 
    &  = f + u \qquad &&\text{in } \Omega, \\
    \label{eq:optimiality-system-b}
    y & = g \qquad && \text{on }\Gamma,
\end{alignat}
\end{minipage}
    \begin{minipage}{0.45\textwidth}
\begin{alignat}{2}
\label{eq:optimiality-system-c}
    - \Delta z -\bm\beta \cdot \nabla z 
    & = y_d - y\qquad && \text{in } \Omega,\\
    \label{eq:optimiality-system-d}
    z & = 0\qquad && \text{on }\Gamma,\\
    \label{eq:optimiality-system-e}
    z - \gamma u 
    & = 0 \qquad && \text{in }\Omega .
\end{alignat}
\end{minipage}
\end{subequations}

The reminder of the paper is organized as follows. In Section \ref{ch:preliminaries}, we state admissibility conditions on the family of computational domains that approximate the true domain, introduce the transferring path technique, and set some notation. Additionally, we provide assumptions on the distance between $\Gamma$ and the computational boundary $\Gamma_h$. The HDG scheme is presented in Section \ref{ch:HDG_formulation} and its well-posedness is proved. Subsequently, in Section \ref{ch:error_analyisis} we provide the \textit{a priori} error estimates of the method and present numerical experiments in Section \ref{ch:numerical_experiments}. We end with concluding remarks in Section \ref{ch:conclusions}.

\section{Preliminaries}\label{ch:preliminaries}

In this section we introduce the notation associated with the computational domain and the family of paths that will allow us to transfer the boundary data from $\Gamma$ to the computational boundary $\Gamma_h$. To that end, we will consider the setting specified in \cite{SaSaSo2022} and establish a set of assumptions under which our analysis is valid.

\paragraph{Admissible triangulations.}
\label{ch:computational_domain_and_transferring_paths}

Given a domain $\Omega$ and a discretization parameter $h>0$, we denote by $\Omega_h$ an open polygonal/polyhedral computational domain, with boundary $\Gamma_h$, triangulated by a simplicial mesh $\mathcal{T}_h$ of meshsize $h$. For a simplex $K$, we denote its outward unit normal vector by $\bm n_K$, writing $\bm n$ instead of $\bm n_K$ when there is no confusion. Similarly, for a facet $e$, we write $\bm n$ instead of $\bm n_e$ to refer to its normal vector. We also consider, by simplicity, that the triangulation does not have hanging nodes. The set of facets and boundary facets of $\mathcal{T}_h$ are denoted by $\mathscr{E}_h$ and $\mathscr{E}_h^{\partial}$, respectively. The set $\Omega_h^c := \Omega \backslash \overline{\Omega_h}$ refers to the non-meshed region. We say the family $\{(\Omega_h,\mathcal{T}_h)\}_{h>0}$ is {\it admissible} if each member $(\Omega_h,\mathcal{T}_h)$ satisfies the following conditions: 

\begin{enumerate} [(a)]
    \item $\Omega_h \subset \Omega$;
    \item $\mathcal{T}_h$ is uniformly shape-regular, that is, there exists 
    $\hat{\gamma} >0$, independent of $h$, such that $h_K\leq 
    \hat{\gamma}\rho_K$, where $\rho_K$ is the radius of the largest ball 
    contained in $K$ and $h_K<h$ is the diameter of $K$;
    \item there exists a bijective mapping $\phi: \Gamma_h \rightarrow \Gamma$;
    \item for every $K\in \mathcal{T}_h$ such that $K\cap\Gamma_h \neq \varnothing$, it holds that $\max\{{\rm dist}(\boldsymbol x,\overline{\boldsymbol x}): \boldsymbol x \in K\cap\Gamma_h \text{ and } \overline{\boldsymbol x} \in \Gamma\} = \mathcal O (h_K)$.
\end{enumerate}
Moreover,
\begin{enumerate} [(e)]
        \item for every $\epsilon>0$ there exists a pair $(\Omega_h,\mathcal T_h)$ such that $\lambda(\Omega \setminus \Omega_h)<\epsilon$, where $\lambda(\cdot)$ denotes the Lebesgue measure.
\end{enumerate}

Let us provide a concise explanation of these conditions. If condition (a) is not met, indicating that $\Omega_h$ is not fully contained within $\Omega$, the analysis presented in this manuscript can be adjusted to accommodate such scenarios, assuming that the underlying partial differential equation remains valid outside of $\Omega$. Condition (c) is required for the analysis presented in this paper, but for computational implementation, the mapping $\phi$ does not necessarily have to be bijective as long as (d) is satisfied. Condition (d) is imposed to ensure that the distance between $\Gamma$ and an element $K$ that intersects the boundary $\Gamma_h$  is at most of order $h_K$.

\paragraph{Spaces, mesh dependent inner-products and norms.} In this paper we 
will make use of the usual notation associated Sobolev spaces, i.e. given a domain $D$ in $\mathbb{R}^n$ and $s$ a non-zero real number, we denote $H^s(D)$ the usual Hilbert space equipped with the inner product $(\cdot,\cdot)_{s,D}$. The induced norm and semi-norm are denoted by $\|\cdot\|_{s,D}$ and  $|\cdot|_{s,D}$, respectively.  If $s=0$, we just write $(\cdot, \cdot)_D$ and $\|\cdot\|_{D}$ as usual. The spaces for vector-valued functions will be boldfaced, for instance,  $\boldsymbol{H}^{s}(D) := [H^{s}(D)]^{d}$. For a Lipschitz curve $S$  in $\mathbb{R}^n$, we employ the same notation, but the inner product is denoted $\langle\cdot,\cdot\rangle_S$. 

For a scalar-valued function $\eta$ and $\zeta$, we define
\begin{equation*}
    (\eta,\zeta)_{\mathcal{T}_h}
    := \sum_{K\in\mathcal{T}_h}(\eta,\zeta)_K\quad\text{ and } \quad
    \langle \eta, \zeta \rangle_{\partial\mathcal{T}_h}
    := \sum_{K\in\mathcal{T}_h} \langle \eta, \zeta \rangle_{\partial K}.
\end{equation*}
Vector-valued functions are boldfaced and for $\bm \eta$ and $\bm\zeta$, we write
\begin{equation*}
    (\bm\eta,\bm\zeta)_{\mathcal{T}_h} 
    \,:=\, \sum_{i=1}^d(\eta_i,\zeta_i)_{\mathcal{T}_h}\quad\text{ and }\quad\langle\bm\eta,\bm\zeta\rangle_{\partial\mathcal{T}_h}\,:=\,\sum_{i=1}^d\langle\eta_i,\zeta_i\rangle_{\partial\mathcal{T}_h}\,.
\end{equation*}
These inner products induce the norms
\begin{equation*}
    \norm[0]{\cdot}_{\mathcal{T}_h} 
    := \left(\sum_{K\in\mathcal{T}_h} \norm[0]{\cdot}_{K}^2\right)^{1/2},\quad
    \norm[0]{\cdot}_{\partial\mathcal{T}_h} 
    := \left(\sum_{K\in\mathcal{T}_h}\norm[0]{\cdot}_{\partial K}^2\right)^{1/2}, \quad\text{ and }\quad
    \norm[0]{\cdot}_{\Gamma_h} 
    := \left( \sum_{e \in \mathscr{E}_h^{\partial}} \norm[0]{\cdot}_{e}^2 \right)^{1/2}.
\end{equation*}
In addition, for $w > 0$, we write $\norm[0]{\eta}_{\partial\mathcal{T}_h,w} := \norm[0]{w^{1/2} \eta}_{\partial\mathcal{T}_h}$ and $\norm[0]{\eta}_{\Gamma_h,w} := \norm[0]{w^{1/2} \eta}_{\Gamma_h}$. Finally, to avoid proliferation of constants, we will write $a \lesssim b$ instead of $a \leq C b$, when $C$ is a constant independent of $h$.
 
\paragraph{Transferring paths and polynomial extrapolation.} Since the problem will be solved in $\Omega_h$, we must specify a suitable boundary data on the computational boundary $\Gamma_h$. To this end, we consider the idea proposed by \cite{paper2} and transfer the boundary data $g$ from $\Gamma$ to $\Gamma_h$ by a line integration of the gradient through transferring paths, which originates the name of the method. More precisely, let $e\in\mathscr{E}_h^{\partial}$. For each $\bm x \in e$, we define $l(\bm x) := |\phi(\bm x) - \bm x|$ and denote by $\boldsymbol{m}(\bm x)$ the unit tangent vector to the segment joining $\bm x$ and $\phi(\bm x)$. We notice that, for an admissible triangulation, $l(\bm x) \leq C h$ with $C>0$ independent of $h$. A line integration over this segment allows us to transfer the boundary data. In fact, if a given function $g_v$ is the trace of a function $v$ and $\bm q:=-\nabla v$, there holds
\begin{align}\label{def:transf}
    g_v(\bm x) 
    & = g_v(\bm \phi(\bm x)) 
    + \int_0^{l(\bm x)} (\bm{q} \cdot \bm{m})(\bm x(s)) \text{d}s ,
\end{align}
with $\bm x(s) = (\bm \phi({\bm x}) - \bm x) s / l(\bm x) + \bm x$, $s\in [0,l(\bm x)]$. Now, at the discrete level, a polynomial approximation $\bm q_h$ of $\bm q$ will be available only inside $\Omega_h$. Hence, we will extrapolate $\bm q_h$ to the segment $[\bm x ,\phi(\bm x)]$ in order to compute the above integral. This is why we define the {\it extrapolation patch} as
\begin{equation*}
    K_{ext}^{e} 
    := \left\{\bm x + s \bm m(\bm x) : 0 \leq s \leq l(\bm x), \bm x\in e\right\}.
\end{equation*}
Since $\bm \phi$ is a bijection, we observe that $\Omega_h^c = \cup_{e\in\mathscr{E}_h^\partial} K_{ext}^{e}$. We highlight that a bijection $\bm \phi$ can be constructed in several ways. For instance, a particular construction in two dimensions can be found in \cite{cockburn2012solving}. It is also possible to use the closest-point projection as long as it is unique. 

We denote by $H_e^{\perp}$ the largest distance of a point in $K_{ext}^e$ to 
the plane determined by $e$ and by $h_e^{\perp}$ the distance between $e$ and 
the vertex of $K^{e}$ opposite to $e$. We set $r_e = H_e^{\perp} 
/ h_e^{\perp}$ and define 
\begin{equation*}
\begin{array}{c}
    \displaystyle
    R := \max_{e\in\mathscr{E}_h^{\partial}}r_e.
\end{array}
\end{equation*}
Additionally, we assume a \textit{local proximity condition}, which states that $|\boldsymbol{x} - \bm \phi(\boldsymbol{x})| \leq C h^{ \delta}$. This condition characterizes the closeness between the computational and physical domains. This condition can be equivalently written as
\begin{equation}\label{eq:prox-cond}
    \max_{e \in \mathcal{E}_h^{\partial}} H_e^{\perp} \leq C h^{1+\delta}.
\end{equation}
On the other hand, given a polynomial $\bm q$ defined on a boundary element $K^e\in \mathcal{T}_h$ such that $e = \overline{K^e}\cap\Gamma_h$, $E(\bm{p})$ denotes its extrapolation to $K_{ext}^{e}$. In this direction, it is also useful to introduce the function
\begin{equation}\label{def:Lambda}
    \Lambda^{\bm{p}}(\bm x)
    := l^{-1}(\bm x) \int_{0}^{l(\bm x)}\left(\bm{p}(\bm x) - \bm E(\bm{p})(\bm x(s)) \right)\cdot\bm{m} (\bm x) \dif{s}
\end{equation}
and the norm
\begin{equation}\label{eq:trip-norm}
    \tnorm{\boldsymbol{p}}_e 
    := \left(\int_e \int_0^{l(\boldsymbol{x})} |\boldsymbol{p}(\boldsymbol{x} + t \boldsymbol{m}(\bm x(s)))|^2 \dif{t} \dif{s} \right)^{1/2},
\end{equation}
where $e \in \mathcal{E}_h^{\partial}$ and $\boldsymbol{p}$ is sufficiently smooth for the norm to be properly defined. 

Under certain assumptions on the transfer paths associated to the vertices of the boundary facet $e$, the norms $  \tnorm{\cdot}_e $ and $\|\cdot \|_{K_{\rm ext}^e}$ are equivalent. More precisely, let $\bm m(\bm v_i)$ the unit tangent vector of the transfer path associated with the vertices $\bm v_i$ ($i=1,2$) of the edge $e$. These norms are equivalent (see \cite[Lemma 3.4]{OyZuSo2019}) if
\begin{enumerate}
\item ${\bm m}(\bm v_{1})\cdot {\bm m}(\bm v_{2})\geq 0$,
\item there exists a constant $\beta_{e}$, independent of $h$, such that 
$\bm m (\bm x)\cdot {\bm n}_{e}\geq \beta_{e}>0$ for all ${\bm x} \in e$; and
\item  ${\bm m}({\bm v_{1}})\cdot \left({\bm m}({\bm v_{2}})\right)^{\perp}\geq 0$, with 
$\left(\bm m({\bm v_{2})}\right)^{\perp}$ being the vector obtained from 
$\bm m({\bm v_{2})}$ through a counterclockwise rotation by $\pi/2$ about the origin.
\end{enumerate}
In the three-dimensional case, these norms are also equivalent under similar assumptions involving the three vertices. We omit details and refer to \cite[Lemma A.1]{OySoZu2020}.

In this way, we introduce the following lemma as an extension of \cite[Lemma 5.2]{paper2} to help us to quantify the extrapolation error on a transferring segment. 

\begin{lemma}
\label{lem:Lambdas}
If the aforementioned over $\bm m(\bm v_i)$ hold, then
\begin{subequations}
    \label{ineq:Lambdas}
	\begin{align}
        \label{ineq:Lambdas-a}
		\norm[0]{l^{1/2} \Lambda^{\bm{p}} \rvert_{K_e} }_e 
        & \leq \frac{\beta_e^{-1/2}}{\sqrt{3}} r_e \norm[0]{h_e^\perp  \nabla \boldsymbol{p} }_{K_\mathrm{ext}^e} \qquad \forall \bm{p} \in [H^1(K^e_\mathrm{ext})]^d,
        \\
        \label{ineq:Lambdas-b}
		\norm[0]{ l^{1/2} \Lambda^{\bm{p}}\rvert_{K^e}  }_e 
        & \leq \frac{1}{\sqrt{3}} r^{3/2}_e C_{ext}^e C_{inv}^e \norm[0]{\bm{p}}_{K_e} \qquad \forall \bm{p} \in [\mathbb{P}_k(K^e)]^d,
	\end{align}
\end{subequations}
where 
 \begin{equation*}
	C^e_{ext} := \dfrac{1}{\sqrt{r_e}} \sup_{\boldsymbol{\chi}\in [\mathbb{P}_k(K^e)]^d} \dfrac{\|\boldsymbol{E}(\boldsymbol{\chi}) \|_{K_{\rm ext}^e}}{\|\boldsymbol{\chi} \|_{K^e}}.
\end{equation*}
\end{lemma}
\begin{proof}
    See appendix \ref{sec:proof-lambdas}.
\end{proof}
The constants $C^e_{ext}$ and $C^e_{inv} $ are indeed independent of 
$h$, but they depend on the shape regularity constant and on the 
polynomial degree $k$ \cite{paper2}.

\paragraph{Closeness assumptions.}
We state a set of assumption quantifying how close $\Gamma$ and $\Gamma_h$ 
must be in order to ensure well-posedness and optimal convergence of the 
method. For every facet $e$ of $\mathscr{E}^{\partial}_h$, we suppose that

\begin{minipage}{0.45\textwidth}
\begin{enumerate}[({A}.1),series=mylist]
    \item $\beta_e^{-2} r_e^3 (C^e_{ext})^{2} (C^e_{inv})^{2} 
    \leq 1/32$, \label{Assumption:1}
    \item $r_e h_e^{\perp} (\tau_1 - \dfrac{1}{2} \bm\beta \cdot \bm n) 
    \leq 1/4 $,\label{Assumption:2}
    \item $r_e h_e^{\perp} (\tau_2 + \dfrac{1}{2} \bm\beta \cdot \bm n)
    \leq 1/4$,\label{Assumption:2'}
\end{enumerate}
\end{minipage}
\begin{minipage}{0.45\textwidth}
\begin{enumerate}[resume*=mylist] 
    \item $\displaystyle 2 \max_{x\in e} l(\bm x) (\tau_1 - \dfrac{1}{2} \bm\beta\cdot\bm n) 
    \leq 1/4.$\label{Assumption:4}
    \item $\displaystyle 2 \max_{x\in e} l(\bm x) (\tau_2 + \dfrac{1}{2} \bm\beta \cdot \bm n) 
    \leq 1/4,$\label{Assumption:4'}
\end{enumerate}
\end{minipage} 

\noindent where $\tau_1$ and $\tau_2$ are the stabilization parameters of the HDG method and $\beta_e$ is a constant independent of the discretization parameter $h$. Let us comment on the feasibility of the above conditions. If $\Omega$ is polygonal/polyhedral, then $r_e\,=\,0$ and all the assumptions hold true.  If ${\rm dist}(\Gamma_h,\Gamma)$ is of order $h^{\delta+1}$, they are also true for all $\delta>0$. This is the case, for example, when $\Gamma_h$ is a piece-wise linear interpolation of $\Gamma$ since $\delta = 1$. On the other hand, the TPM can handle more general scenarios where $\Omega$ is embedded in a background triangulation and $\Omega_h$ is constructed by the union of the elements of the triangulation lying completely inside of $\Omega$, which means $1 > \delta > 0$. We observe that\ref{Assumption:2} - \ref{Assumption:4'} are still satisfied for $h$ sufficiently small, but \ref{Assumption:1} cannot be guaranteed in general. However, numerical experiments suggest that the method is still optimal even for the case $\delta=0$. In addition, we decompose  $\bm m =(\bm m \cdot \bm n) \bm n+ \bm t$ with $\bm t:=\bm m-(\bm m \cdot \bm n) \bm n$ and we ask the tangential component $\bm t$ to satisfy
\begin{enumerate}[resume*=mylist] 
    \item $r_e \norm[0]{\bm t}_{L^{\infty}(\Gamma_h)}^2
    \leq \dfrac{1}{98} \beta_e C_{tr}^{-2} ,$\label{Assumption:t}
\end{enumerate}
where $C_{tr} > 0$ is the constant obtained from the discrete trace inequality (see \cite[Lemma 1.46]{di2011mathematical}). This condition states that, from the theoretical point of view,  the transfer path must be constructed in such a way that $\bm m$ does not deviate too much from the normal to the facet.

\paragraph{$L^2$ and HDG projection.} Let us first recall the classical $L^2$ projection into $M_h$ denoted by $P_M$. On each $K \in\mathcal{T}_h$ it satisfies (cf. \cite[Lemma 1.58 and Lemma 1.59]{di2011mathematical}), for $v \in H^{l+1}(K)$, 
\begin{subequations}
\begin{align}
    |v - P_M v|_{m,K}
    & \lesssim 
    h^{l+1-m} |v|_{l+1,K} \quad \forall m \in \{0,\ldots,k\},\\
    \label{eq:proj-L2-faces-estimate}
    \norm[0]{v - P_M v}_{\partial K}
    & \lesssim 
    h^{l+1/2} |v|_{l+1,K}.
\end{align}
\end{subequations}
We also consider the HDG projection introduced in \cite{chen2012analysis}, 
which is a projection into the product space $\bm{V}_h\times W_h$. Given 
$(\bm{q}, y)\in \bm{V}_h\times W_h$, it is defined by
$\Pi_h(\bm{q},y) := (\bm\Pi_V\bm{q},\Pi_W y)$,
where $(\bm\Pi_{\bm{V}}\bm{q},\Pi_W y)$ is the only element satisfying, for 
all $K\in\mathcal{T}_h$,
\begin{subequations}
\label{eq:proj_hdg}
\begin{align}
    \label{eq:proj_hdg_a}
    (\bm\Pi_{\bm{V}}\bm{q} + \bm\beta \Pi_W y, \bm{s})_{K} 
    & = (\bm{q} + \bm\beta y, \bm{s})_{K} &&\forall \bm{s}\in \left[\mathbb{P}_{k-1}(K)\right]^{d},\\
    \label{eq:proj_hdg_b}
    (\Pi_W y,t)_{K}
    & = (y,t)_{K} &&\forall t \in \mathbb{P}_{k-1}(K),\\
    \label{eq:proj_hdg_c}
    \langle \bm\Pi_{\bm{V}} \bm{q} \cdot \bm{n} + \bm\beta \cdot \bm n P_My + \tau_1 \Pi_W y, \mu\rangle_e 
    & = \langle \bm{q} \cdot \bm{n} + \bm\beta \cdot \bm n y + \tau_1  y, \mu \rangle_{e} && \forall \mu \in \mathbb{P}_{k}(e), \quad \forall  e \subseteq \partial K,
\end{align}
\end{subequations}
where $\tau_1$ is the positive stabilization parameter of the HDG scheme \eqref{equations:hdg} defined in $\partial \mathcal{T}_h$. If $\tau_1$ is chosen independent of $h$ and such that $\big(\tau_1 - \frac{1}{2} \bm\beta(\bm x) \cdot \bm n(\bm x)\big)>0$, for all $\bm x\in\partial K$, then this projection is well-defined 
\cite{chen2012analysis}. Moreover, if $\bm q \in {\bm H}^{l+1}(K)$ and 
$y \in H^{l+1}(K)$, with $l\in [0,k]$,
\begin{subequations}
 \label{eq:hdg_proj_err}
\begin{align}
    \label{eq:hdg_proj_err_a}
    \norm[0]{\bm \Pi_{\bm V} \bm q - \bm q}_{K}
    & \lesssim
    h_K^{l+1} |\bm q|_{l+1,K} + h_{K}^{l+1} |y|_{l+1,K},\\
    \label{eq:hdg_proj_err_b}
    \norm[0]{\Pi_{W}y\,-\,y}_{K}
    & \lesssim 
    h^{l+1}_K |\bm q|_{l+1,K} + h_K^{l+1} |y|_{l+1,K}.
\end{align}
\end{subequations}
In turn, for the convection-diffusion optimal control problem, we will use 
the projection operator associated to the dual problem, introduced in 
\cite{chen2018hdg}. It is a projection into the product space $\bm V_h \times 
W_h$ such that, given $(\bm p, z)\in \bm V_h\times W_h$, it is defined by 
$\widetilde \Pi_h(\bm p, z) := (\widetilde{\bm\Pi}_{\bm V} \bm 
p, \widetilde\Pi_W z)$, where, on each element $K \in \mathcal T_h$, 
\begin{subequations}
\label{eq:proj_hdg_d}
\begin{align}
    \label{eq:proj_hdg_d_a}
    (\widetilde{\bm\Pi}_{\bm{V}}\bm{p}\,-\,\bm\beta\,\widetilde\Pi_W z,\bm{s})_{K} 
    & = (\bm{p} - \bm\beta z, \bm{s})_{K} && \forall \bm{s} \in \left[\mathbb{P}_{k-1}(K)\right]^{d},\\
    \label{eq:proj_d_hdg_b}
    (\widetilde\Pi_W z,t)_{K}
    & = (z,t)_{K} &&\forall t \in \mathbb{P}_{k-1}(K),\\
    \label{eq:proj_hdg_d_c}
    \langle \widetilde{\bm\Pi}_{\bm{V}} \bm{p} \cdot \bm{n} - \bm\beta \cdot \bm n P_Mz + \tau_2 \widetilde\Pi_W z, \mu \rangle_e 
    & = \langle \bm{p} \cdot \bm{n} - \bm\beta \cdot \bm n z + \tau_2 z, \mu \rangle_{e} && \forall \mu \in \mathbb{P}_{k}(e), \quad \forall e \subseteq \partial K,
\end{align}
\end{subequations}
where $\tau_2$ is the positive stabilization parameter of the HDG scheme \eqref{equations:hdg} defined in $\partial \mathcal{T}_h$, chosen to be independent of $h$ and such that $\big(\tau_2 + \frac{1}{2} \bm\beta(\bm x) \cdot \bm n(\bm x) \big)>0$, for all $\bm x\in\partial K$. This projection is of the same nature as $\Pi_h$. Therefore, it is well-defined and, if $\bm p \in \bm H^{l+1}(K)$ and $z\in H^{l+1}(K)$, with $l\in[0,k]$, then
\begin{subequations}
\begin{align}
    \label{eq:hdg_proj_err_d_a}
    \norm[0]{\widetilde{\bm \Pi}_{\bm V} \bm p - \bm p}_{K}
    & \lesssim
    h_K^{l+1} |\bm p|_{l+1,K} + h_{K}^{l+1} |z|_{l+1,K},\\
    \label{eq:hdg_proj_err_d_b}
    \norm[0]{\widetilde\Pi_{W}z - z}_{K}
    & \lesssim
    h^{l+1}_K |\bm p|_{l+1,K} + h_K^{l+1} |z|_{l+1,K}.
\end{align}
\end{subequations}
In the computational boundary $\Gamma_h$ and in the extrapolation region $\Omega_h^c = \cup_{e\in\mathscr{E}_h^\partial} K_{ext}^{e}$, the HDG projection satisfies (cf.  \cite[Lemma 3.8]{paper2}),
\begin{subequations} \label{eq:Ipn_Irn}
\begin{align}
    \label{eq:Ipn_Irn_a}
    \norm[0]{\bm \Pi_{\bm V}\bm q - \bm q}_{\Gamma_h,h^{\perp}}
    & \lesssim
    h^{k+1} |\bm q|_{l+1,\Omega}
    + h^{k+1} |v|_{l+1,\Omega},\\
    \label{lemma:partial_Ip_Ir}
    \norm[0]{\partial_{\bm n}((\bm \Pi_{\bm V}\bm q - \bm q) \cdot \bm n)}_{\Omega_h^c,(h^{\perp})^2}
    & \lesssim
    R^{1/2} \norm[0]{\bm \Pi_{\bm V} \bm q - \bm q}_{\mathcal{T}_h} + \left(1 + R^{1/2}\right) h^{k+1} |\bm q|_{l+1, \Omega}.
\end{align}
\end{subequations}
Furthermore, $\widetilde{\Pi}_h$ also satisfies analogous estimates.


\section{HDG formulation}\label{ch:HDG_formulation}

In order to present the HDG scheme, we state the strong mixed formulation of the equation posed in the computational subdomain $\Omega_h$,  which is given by

\begin{subequations}
    \begin{minipage}{0.45\textwidth}
\begin{alignat}{3}
    \label{eq:pde_a}
    \bm{q} + \nabla y & = 0 \qquad  &\text{in }\Omega_h,\\
    \label{eq:pde_b}
    \nabla \cdot(\bm{q} + \bm\beta y) & = f + \alpha z\qquad&\text{in }\Omega_h,\\
    \label{eq:pde_c}
    y & = \varphi_1 \qquad &\text{on }\Gamma_h,
\end{alignat}
\end{minipage}
\qquad 
\begin{minipage}{0.45\textwidth}
    \begin{alignat}{3}
    \label{eq:pde_d}
    \bm{p} + \nabla z & = 0 \qquad& \text{in } \Omega_h,\\
    \label{eq:pde_e}
    \nabla\cdot(\bm{p} - \bm\beta z) & =  y_d - y\qquad& \text{in }\Omega_h,\\
    \label{eq:pde_f}
    z & = \varphi_2 \qquad& \text{on } \Gamma_h,
\end{alignat}
\end{minipage}
\end{subequations} \\

\noindent where $\varphi_1$ and $\varphi_2$ can be obtained by integrating \eqref{eq:pde_a} and \eqref{eq:pde_d} along the transferring paths and $\alpha := \gamma^{-1}$. More precisely, for any $\bm x\in e$, $e\in\mathcal{E}_h^\partial$ and $\overline{\bm x}\in \partial\Omega$, we deduce that (cf. \eqref{def:transf})
\begin{align}
    \varphi_{1}(\bm x) 
    & = g(\overline{\bm x}) + \int_0^{l(\bm x)} \bm{q} \cdot \bm{m}(\bm x(s)) \text{d}s \quad {\rm and} \quad 
    \varphi_{2}(\bm x) 
    = \int_0^{l(\bm x)}\bm{p} \cdot \bm{m}(\bm x(s)) \text{d}s,
    \label{def:phi}
\end{align}
where $\bm x(s) = (\overline{\bm x} - \bm x) s /l(\bf x) + \bm x$, $s \in [0,l(\bf x)]$. The HDG method seeks an approximation $(\bm q_h, y_h, \widehat{y}_h, \bm p_h, z_h, \widehat{z}_h)$ of the exact solution $(\bm q, y, y\big|_{\mathscr{E}_{h}}, \bm p, z, z\big|_{\mathscr{E}_{h}})$ in the space $\bm{V}_h \times W_h \times M_h \times \bm{V}_h \times W_h \times M_h$ defined by
\begin{subequations}
\label{eq:discrete_spaces}
\begin{align}
    \bm{V}_h 
    &:= \left\{ \bm{v}\in [L^2(\mathcal{T}_h)]^d : v\big|_K \in[\mathbb{P}_{k}(K)]^d, \quad \forall K \in \mathcal{T}_h \right\},\\
    W_h 
    & := \left\{ w \in L^{2}(\mathcal{T}_h) : w \big|_K\in\mathbb{P}_{k}(K), \quad \forall K \in \mathcal{T}_h \right\},\\
    M_h 
    & := \left\{ \mu \in L^2(\mathscr{E}_h) : \mu \big|_e \in \mathbb{P}_{k}(e), \quad \forall e \in \mathscr{E}_{h}\right\}.
\end{align}
\end{subequations}
such that
\begin{subequations}\label{equations:hdg}
\begin{align}
    \label{eq:hdg_a}
    (\bm{q}_h, \bm{r}_1)_{\mathcal{T}_h}
    - (y_h, \nabla \cdot \bm{r}_1)_{\mathcal{T}_h}
    + \langle \widehat{y}_h, \bm{r}_1 \cdot \bm{n} \rangle_{\partial\mathcal{T}_h} & = 0,\\
    \label{eq:hdg_b}
    - (\bm{q}_{h} + \bm\beta y_h, \nabla w_1)_{\mathcal{T}_h} + \langle\widehat{\bm{q}}_h \cdot \bm{n} + \bm\beta \cdot \bm n \widehat y_h, w_1\rangle_{\partial\mathcal{T}_h}
    - (\alpha z_h, w_1)_{\mathcal{T}_h}
    & = (f, w_1)_{\mathcal{T}_h},\\
    \label{eq:hdg_c}
    (\bm{p}_h,\bm{r}_2)_{\mathcal{T}_h}
    - (z_h, \nabla \cdot \bm{r}_2)_{\mathcal{T}_h}
    + \langle \widehat{z}_h, \bm{r}_2 \cdot \bm{n} \rangle _{\partial\mathcal{T}_h}
    & = 0,\\
    \label{eq:hdg_d}
    - (\bm{p}_{h} - \bm\beta z_h, \nabla w_2)_{\mathcal{T}_h} + \langle \widehat{\bm{p}}_h \cdot \bm{n} - \bm\beta \cdot \bm n \widehat z_h, w_2 \rangle_{\partial\mathcal{T}_h}
    + (y_h, w_2)_{\mathcal{T}_h}
    & = (y_d,w_2)_{\mathcal{T}_h}\,,\\
    \label{eq:hdg_e}
    \langle \widehat{\bm{q}}_h \cdot \bm{n} + \bm\beta \cdot \bm n \widehat y_h, \mu_1 \rangle_{\partial \mathcal{T}_h \backslash \Gamma_h}
    & 
    = 0,\\
    \label{eq:hdg_f}
    \langle \widehat{\bm{p}}_h \cdot \bm{n} - \bm\beta \cdot \bm n \widehat z_h, \mu_2 \rangle_{\partial\mathcal{T}_h \backslash \Gamma_h}
    & 
    = 0,\\
    \label{eq:hdg_i}
    \langle \widehat{y}_h, \mu_1 \rangle_{\Gamma_h} 
    & =
    \langle \varphi_1^{h}, \mu_1 \rangle_{\Gamma_h},\\
    \label{eq:hdg_j}
    \langle \widehat{z}_h, \mu_2 \rangle_{\Gamma_h} 
    & =
    \langle \varphi_2^{h}, \mu_2 \rangle_{\Gamma_h},
\end{align}
for all $(\bm{v}_1, w_1, \mu_1, \bm{v}_2, w_2, \mu_2) \in \bm{V}_h \times W_h \times M_h \times \bm{V}_h \times W_h \times M_h$. The numerical traces are defined as follows
\begin{align}\label{eq:hdg_gh}
    \bm{\widehat{q}}_h 
    = \bm{q}_h + \tau_1 (y_h - \widehat{y}_h) \bm{n} \qquad \text{and} \qquad \bm{\widehat{p}}_h 
    = \bm{p}_h + \tau_2 (z_h - \widehat{z}_h) \bm{n},
\end{align}
where $\tau_1$ and $\tau_2$ are stabilization parameters. Now, we observe that the true boundary data on $\Gamma_h$ (cf. \eqref{def:phi}) depend on the unknowns $\bm{p}$ and $\bm{q}$. Then, we approximate $ \varphi_{1}$ and $\varphi_{2}$ by $\varphi_{1}^h$ and $\varphi_{2}^h$ as follows. Let $\bm x \in \mathcal{E}_h^\partial$ and set
\begin{align}\label{def:varphi12}
    \varphi_{1}^h(\bm x) 
    := g(\overline{\bm x}) + \int_0^{l(x)} \bm E(\bm{q}_h) \cdot \bm{m}(\bm x(s)) \dif{s}
    \qquad \text{and} \qquad
    \varphi_{2}^h(\bm x)
    := \int_0^{l(x)} \bm E(\bm{p}_h) \cdot \bm{m}(\bm x(s)) \dif{s},
\end{align}
\end{subequations}
where we recall the $E(\bm{q}_h)$ and $E(\bm{p}_h)$ are the local 
extrapolation defined in the previous section. For our theoretical results, we require the assumptions used in \cite[Section 3]{chen2018hdg} for the stabilization parameters:

\begin{minipage}{0.45\textwidth}
\begin{enumerate}[({B}.1),series=mylist]
    \item $\tau_1$ is a piecewise constant on $\partial \mathcal T_h$, \label{Assumption:B-1}
    \item $\tau_1 = \tau_2 + \bm\beta \cdot \bm n$, \label{Assumption:B-2}
\end{enumerate}
\end{minipage}
\begin{minipage}{0.45\textwidth}
\begin{enumerate}[resume*=mylist] 
    \item $\min(\tau_1 - \dfrac{1}{2} \bm\beta \cdot \bm n)|_{\partial K}
    > 0 \quad \forall K \in \mathcal{T}_h$, \label{Assumption:B-4}
\end{enumerate}
\end{minipage}\\

Note that if \ref{Assumption:B-2} and \ref{Assumption:B-4} hold, then
\begin{equation}\label{eq:B-4'}
    \min(\tau_2 + \frac{1}{2} \bm\beta \cdot \bm n)|_{\partial K}
    > 0 \quad \forall K \in \mathcal{T}_h.    
\end{equation}
To analyze the well-posedness of this HDG formulation, we introduce the following notation:
\begin{subequations}
\begin{align}
\begin{split}
    E := \gamma \norm[0]{ \bm{q}_h}_{\mathcal{T}_h}^2
    + \norm[0]{\bm{p}_h}_{\mathcal{T}_h}^{2}
    + \gamma \norm[0]{(\tau_1 - \dfrac{1}{2} \bm\beta \cdot \bm n)^{1/2} (y_h -\widehat{y}_h)}_{\partial\mathcal{T}_h}^{2}
    + \norm[0]{(\tau_2 + \dfrac{1}{2} \bm\beta \cdot \bm n)^{1/2} (z_h - \widehat{z}_h)}_{\partial\mathcal{T}_h}^2
    \end{split}\label{def:E}
\end{align}
\begin{equation}
    E_\varphi := \gamma \norm[0]{l^{-1/2} \varphi_1^h}_{\Gamma_h}^2
    + \norm[0]{l^{-1/2} \varphi_2^h}_{\Gamma_h}^2. \label{def:Evarphi}
\end{equation}
\end{subequations}
In other words, $E$ quantifies the energy of the HDG solution in the computational domain $\Omega_h$, whereas $E_\varphi$ is related to the approximation of the boundary data. From now on, assumptions (A) and (B) will be assumed to be valid.

\begin{lemma}\label{lemma:EU}
  The HDG scheme \eqref{equations:hdg} is well-posed.  
\end{lemma}

We will come back to the proof after Section \ref{ch:error_analyisis}.

\section{Error analysis}\label{ch:error_analyisis}

As usual in the error analysis of these types of methods \cite{cockburn2010projection}, we first decompose the error as the sum of the projection of the error and the projection error. The latter will be controlled by the properties in Section \ref{ch:preliminaries}, while for the former we will employ an energy argument (Section \ref{ch:energy_argument}) for the mixed variables and a duality argument (Section \ref{ch:duality_argument}) for the primal variables. 

\subsection{Energy argument}\label{ch:energy_argument}

We introduce the following notation for the error, the projection of the error, and the projection error, respectively:
\begin{subequations}
\begin{alignat}{10}
     & \bm e^{\bm q} := \bm{q} - \bm{q}_h \quad&, \quad  
     & e^y := y - y_h \quad&, \quad 
     & e^{\widehat{y}} := y - \widehat{y}_h&, \\
     & \bm\varepsilon_{\bm q} = \bm\Pi_{\bm{V}}\bm{q} -\bm{q}_h \quad&, \quad & \varepsilon_y := \Pi_W y - y_h \quad&,\quad &\varepsilon_{\widehat{y}} := P_M y - \widehat{y}_h&, \\
     & I_{\bm q} = \bm{q} - \bm\Pi_{\bm{V}}\bm{q} \quad&, \quad 
     & I_y := y - \Pi_W y \quad&. \quad 
     &&
\end{alignat}
\end{subequations}\label{def:errors}
In the same way, but with the projector $(\widetilde{\bm\Pi}_{\bm V} 
\bm p, \widetilde\Pi_W z)$, we define $\bm e^{\bm p}$, $e^z$, 
$e^{\widehat{z}}$, $\bm\varepsilon_{\bm p}$, $\varepsilon_z$, 
$\varepsilon_{\widehat{z}}$, $ I_{\bm p}$ and $I_z$. To simplify the notation, we define
\begin{subequations}\label{def:mathcal_Epy_Epz_mathcal_Evarphi1,Evarphi2}
\begin{equation}\label{def:mathcal_Epy_Epz}
    \mathcal{E}_{\bm q, y} 
    := \gamma \norm[0]{\bm\varepsilon_{\bm q}}_{\mathcal{T}_h}^2
    + \gamma \norm[0]{(\tau_1 - \dfrac{1}{2} \bm\beta \cdot \bm n)^{1/2}(\varepsilon_y - \varepsilon_{\widehat{y}})}_{\partial \mathcal{T}_h}^{2}, 
    \mathcal{E}_{\bm r, z} := \norm[0]{\bm\varepsilon_{\bm p}}_{\mathcal{T}_h}^{2}
    + \norm[0]{(\tau_2 + \dfrac{1}{2} \bm\beta \cdot \bm n)^{1/2}(\varepsilon_z - \varepsilon_{\widehat{z}})}_{\partial \mathcal{T}_h}^2,
\end{equation}
\begin{equation}
    \mathcal{E}_{\varphi_1} 
    := \gamma \norm[0]{\varphi_1 - \varphi_1^h}_{\Gamma_h,l^{-1}}^2, \quad\quad
    \mathcal{E}_{\varphi_2} 
    := \norm[0]{\varphi_2 -\varphi_2^h}_{\Gamma_h, l^{-1}}^2. \label{def:mathcal_Evarphi1,Evarphi2} 
\end{equation} 
\end{subequations}
We also set $\mathcal{E} := \mathcal{E}_{\bm q, y} + \mathcal{E}_{\bm p, z}$ and $\mathcal{E}_{\varphi} := \mathcal{E}_{\varphi_1} + \mathcal{E}_{\varphi_2}$. The latter allows us to quantify the error in the approximation the boundary data. 

We begin by noting that the projection of the errors satisfied the same HDG 
system \eqref{equations:hdg} but with a different right-hand side. In fact, 
it can be deduced that (see for instance 
\cite{chen2018hdg,henriquez2023control,ZHU20162}):
\begin{subequations}
\label{eq:energy_arg}
\begin{align}
    \label{eq:energy_arg_a}
    ( \bm\varepsilon_{\bm q}, \bm{r}_1)_{\mathcal{T}_h}
    - (\varepsilon_y, \nabla \cdot \bm{r}_1)_{\mathcal{T}_h}
    + \langle \varepsilon_{\widehat{y}}, \bm{r}_1 \cdot \bm{n}\rangle_{\partial \mathcal{T}_h}
    & = - (I_{\bm{q}}, \bm{r}_1)_{\mathcal{T}_h},\\[1ex]
    \label{eq:energy_arg_b}
    - (\bm\varepsilon_{\bm q} + \bm\beta \varepsilon_y, \nabla w_1)_{\mathcal{T}_h}
    + \langle \bm\varepsilon_{\widehat{\bm q}} \cdot \bm{n} + \bm\beta \cdot \bm n \varepsilon_{\widehat y}, w_1\rangle_{\partial \mathcal{T}_h}
    - (\alpha \varepsilon_z, w_1)_{\mathcal{T}_h}
    & = (\alpha I_z, w_1)_{\mathcal{T}_h},\\[1ex]
    \label{eq:energy_arg_c}
    (\bm\varepsilon_{\bm p},\bm{r}_2)_{\mathcal{T}_h}
    - (\varepsilon_z, \nabla \cdot \bm{r}_2)_{\mathcal{T}_h}
    + \langle \varepsilon_{\widehat{z}}, \bm{r}_2 \cdot \bm{n}\rangle_{\partial \mathcal{T}_h}
    &= - (I_{\bm{p}}, \bm{r}_2)_{\mathcal{T}_h},\\[1ex]
    \label{eq:energy_arg_d}
    - (\bm\varepsilon_{\bm p} - \bm\beta \varepsilon_z, \nabla w_2)_{\mathcal{T}_h}
    + \langle \bm\varepsilon_{\widehat{\bm p}} \cdot \bm{n} 
    -\bm \beta \cdot \bm n \varepsilon_{\widehat z}, w_2 \rangle_{\partial \mathcal{T}_h}
    + (\varepsilon_y, w_2)_{\mathcal{T}_h}
    &= -( I_y, w_2)_{\mathcal{T}_h},\\[1ex]
    \label{eq:energy_arg_e}
    \langle \bm \varepsilon_{\widehat{\bm q}} \cdot \bm{n} 
    + \bm \beta \cdot \bm n \varepsilon_{\widehat y}, \mu_1 \rangle_{\partial \mathcal{T}_h \backslash \Gamma_h} 
    &= 0,\\[1ex]
    \label{eq:energy_arg_f}
    \langle \bm \varepsilon_{\widehat{\bm p}} \cdot \bm{n} 
    - \bm \beta \cdot \bm n \varepsilon_{\widehat z}, \mu_2 \rangle_{\partial \mathcal{T}_h \backslash \Gamma_h} 
    &= 0,\\[1ex]
    \label{eq:energy_arg_g}
    \langle \varepsilon_{\widehat{y}}, \mu_1 \rangle_{\Gamma_h} 
    &= \langle \varphi_1 - \varphi_1^h, \mu_1 \rangle_{\Gamma_h},\\[1ex]
    \label{eq:energy_arg_h}
    \langle \varepsilon_{\widehat{z}}, \mu_2 \rangle_{\Gamma_h} 
    &= \langle \varphi_2 - \varphi_2^h, \mu_2 \rangle_{\Gamma_h}
\end{align}
\end{subequations}
$\forall (\bm{r}_1, w_1, \mu_1, \bm{r}_2, w_2, \mu_2) \in \bm{V}_h \times 
W_h \times M_h \times \bm{V}_h \times W_h\times M_h$, with
\begin{align}
    \label{eq:energy_arg_i}
    \bm \varepsilon_{\widehat{\bm q}} \cdot \bm{n} 
    &= \bm \varepsilon_{\bm q} \cdot \bm{n} 
    + \tau_1 (\varepsilon_{y} 
    - \varepsilon_{\widehat{y}}),\\[1ex]
    \label{eq:energy_arg_j}
    \bm \varepsilon_{\widehat{\bm p}} \cdot \bm{n} 
    &= \bm \varepsilon_{\bm p} \cdot \bm{n} 
    + \tau_2 (\varepsilon_{z} 
    - \varepsilon_{\widehat{z}}).
\end{align}
Following exactly the same steps performed in the proof of Lemma \ref{lemma:EU} apply now to the system \eqref{eq:energy_arg}, we find the following identities.

\begin{lemma}\label{lemma:energy_identity} 
It holds
\begin{equation}\label{lemma_eq:energy_identity}
\begin{array}{l}                  
    \mathcal{E}
    = - \gamma (I_{\bm{q}}, \bm \varepsilon_{\bm q})_{\mathcal{T}_h}
    + (I_{z}, \varepsilon_{y})_{\mathcal{T}_h}
    - (I_{\bm{p}}, \bm \varepsilon_{\bm p})_{\mathcal{T}_h}
    -(I_{y}, \varepsilon_z)_{\mathcal{T}_h} 
    + \mathbb{K}_y 
    + \mathbb{K}_z,
\end{array}
\end{equation}
where
\begin{align*}
    \mathbb{K}_y
    & := - \gamma \left\langle \bm \varepsilon_{\bm q} \cdot \bm n 
    + (\tau_1 - \dfrac{1}{2} \bm\beta \cdot \bm n)^{1/2} (\varepsilon_y 
    - \varepsilon_{\widehat y}) 
    + \dfrac{1}{2} \bm\beta \cdot \bm n \varepsilon_y, \varphi_1 - \varphi_1^h \right\rangle_{\Gamma_h},\\
    \mathbb{K}_z
    & := - \left \langle \bm\varepsilon_{\bm p} \cdot \bm n 
    + (\tau_2 + \dfrac{1}{2} \bm\beta \cdot \bm n)^{1/2} (\varepsilon_z 
    - \varepsilon_{\widehat z}) 
    - \dfrac{1}{2} \bm\beta \cdot \bm n \varepsilon_z, \varphi_2 
    - \varphi_2^h \right \rangle_{\Gamma_h}.
\end{align*}
Moreover, on $\Gamma_h$ we have that
\begin{subequations}
\begin{align}
    \label{eq:identity_ehq_ehp_a}
    \bm \varepsilon_{\bm q} \cdot \bm{n}
    & = (\boldsymbol{m} \cdot \boldsymbol{n} l)^{-1} (\varphi_1 - \varphi_1^h)
    + (\bm m \cdot \bm n)^{-1} \big(\Lambda^{I_{\bm q}}
    + \Lambda^{\bm\varepsilon_{\bm q}}
    - I_{\bm{q}} \cdot \bm{m} 
    - \bm \varepsilon_{\boldsymbol{q}} \cdot \bm t\big),\\
    \label{eq:identity_ehq_ehp_b}
    \bm \varepsilon_{\bm p} \cdot \bm{n}
    & = (\boldsymbol{m} \cdot \boldsymbol{n} l)^{-1} (\varphi_2 - \varphi_2^h)
    + (\bm m \cdot \bm n)^{-1} \big(\Lambda^{I_{\bm p}}
    + \Lambda^{\bm\varepsilon_{\bm p}}
    - I_{\bm{p}} \cdot \bm{m} 
    -\bm \varepsilon_{\boldsymbol{p}} \cdot \bm t\big),
\end{align}
\end{subequations}
where we recall the definition of $\Lambda^{\bm p}$ in \eqref{def:Lambda}.
\end{lemma}
\begin{proof}
    Choosing $\bm r_1 = \bm \varepsilon_{\bm q}$, $w_1 = \varepsilon_y$,
    $\bm r_2 = \bm \varepsilon_{\bm p}$, and $w_2 = \varepsilon_z$, in \eqref{eq:energy_arg_a}-\eqref{eq:energy_arg_d}, respectively. 
    Next, applying integration by parts \eqref{eq:energy_arg_b} and \eqref{eq:energy_arg_d}, we get
    \begin{align*}
        \norm[0]{\bm \varepsilon_{\bm q}}_{\mathcal{T}_h}
        - (\varepsilon_y, \nabla\cdot \bm \varepsilon_{\bm q})_{\mathcal{T}_h}
        + \langle\varepsilon_{\widehat{y}}, \bm \varepsilon_{\bm q} \cdot \bm n\rangle_{\partial \mathcal{T}_h}
        & = - (I_{\bm q}, \bm \varepsilon_{\bm q})_{\mathcal{T}_h},\\
        (\nabla \cdot (\bm \varepsilon_{\bm q} 
        +\beta \varepsilon_y), \varepsilon_y)_{\mathcal{T}_h}
        - \langle\bm \varepsilon_{\bm q} \cdot \bm n 
        + \bm \beta \cdot \bm n \varepsilon_y, \varepsilon_y\rangle_{\partial \mathcal{T}_h}
        + \langle\bm\varepsilon_{\widehat{\bm q}}\cdot \bm n 
        + \bm \beta\cdot \bm n \varepsilon_{\widehat{y}}, \varepsilon_h\rangle_{\partial \mathcal{T}_h}
        - \alpha (\varepsilon_z, \varepsilon_y)_{\mathcal{T}_h}
        & = \alpha (I_z, \varepsilon_y)_{\mathcal{T}_h},\\
        \norm[0]{\bm \varepsilon_{\bm p}}_{\mathcal{T}_h}
        - (\varepsilon_z, \nabla\cdot\bm \varepsilon_{\bm p})_{\mathcal{T}_h}
        + \langle\varepsilon_{\widehat{z}}, \bm \varepsilon_{\bm p}\cdot \bm n\rangle_{\partial \mathcal{T}_h}
        &= - (I_{\bm p}, \bm \varepsilon_{\bm p})_{\mathcal{T}_h} \\
        (\nabla \cdot (\bm \varepsilon_{\bm p} 
        - \beta \varepsilon_z), \varepsilon_z)_{\mathcal{T}_h}
        - \langle\bm \varepsilon_{\bm p} \cdot \bm n 
        - \bm \beta \cdot \bm n  \varepsilon_z, \varepsilon_z\rangle_{\partial \mathcal{T}_h}
        + \langle\bm\varepsilon_{\widehat{\bm p}}\cdot \bm n 
        - \bm \beta\cdot \bm n \varepsilon_{\widehat{z}}, \varepsilon_z\rangle_{\partial \mathcal{T}_h}
        + (\varepsilon_y, \varepsilon_z)_{\mathcal{T}_h}
        & = - \alpha (I_y, \varepsilon_z)_{\mathcal{T}_h}.
    \end{align*}
    Adding first and second equations, and the third and the fourth as well, 
    \begin{align*}
        &\norm[0]{\bm \varepsilon_{\bm q}}_{\mathcal{T}_h}^2
        + \langle \varepsilon_{\widehat{y}}, \bm \varepsilon_{\bm q}\cdot \bm n\rangle_{\partial\mathcal{T}_h}
        + \langle(\bm \varepsilon_{\widehat{\bm q}} 
        - \bm \varepsilon_{\bm q})\cdot \bm n
        + \bm \beta \cdot\bm n (\varepsilon_{\widehat{y}}
        - \varepsilon_y), \varepsilon_y\rangle_{\partial\mathcal{T}_h}
        + (\nabla\cdot(\bm \beta \varepsilon_y), \varepsilon_y)_{\mathcal{T}_h}
        - \alpha (\varepsilon_z, \varepsilon_y)_{\mathcal{T}_h}\\
        &\qquad= - (I_{\bm q}, \bm \varepsilon_{\bm q})_{\mathcal{T}_h}
        + \alpha (I_z, \varepsilon_y)_{\mathcal{T}_h},\\
        &\norm[0]{\bm \varepsilon_{\bm p}}_{\mathcal{T}_h}^2
        + \langle \varepsilon_{\widehat{z}}, \bm \varepsilon_{\bm p}\cdot \bm n\rangle_{\partial\mathcal{T}_h}
        + \langle(\bm \varepsilon_{\widehat{\bm p}} 
        - \bm \varepsilon_{\bm p})\cdot \bm n
        - \bm \beta \cdot\bm n (\varepsilon_{\widehat{z}}
        - \varepsilon_z), \varepsilon_z\rangle_{\partial\mathcal{T}_h}
        - (\nabla\cdot(\bm \beta \varepsilon_z), \varepsilon_z)_{\mathcal{T}_h}
        - (\varepsilon_y, \varepsilon_z)_{\mathcal{T}_h}\\
        &\qquad= - (I_{\bm p}, \bm \varepsilon_{\bm p})_{\mathcal{T}_h}
        -(I_y, \varepsilon_z)_{\mathcal{T}_h}.
    \end{align*}
    Then, by \eqref{eq:energy_arg_i}, \eqref{eq:energy_arg_j}, and using the identities $(\bm \beta \varepsilon_y, \nabla \varepsilon_y)_{\mathcal{T}_h} = \dfrac{1}{2}\langle\bm \beta\cdot\bm n \varepsilon_y, \varepsilon_y \rangle_{\partial \mathcal{T}_h}$ and $(\bm \beta \varepsilon_z, \nabla \varepsilon_z)_{\mathcal{T}_h} = \dfrac{1}{2}\langle\bm \beta\cdot\bm n \varepsilon_z, \varepsilon_z\rangle_{\partial \mathcal{T}_h}$, after some algebraic manipulations, we find
    \begin{align*}
        &\norm[0]{\bm \varepsilon_{\bm q}}_{\mathcal{T}_h}^2
        + \norm[0]{(\tau_1 - \dfrac{1}{2} \bm \beta\cdot\bm n)^{1/2}(\varepsilon_y - \varepsilon_{\widehat{y}})}_{\partial\mathcal{T}_h}^2
        + \langle \bm \varepsilon_{\widehat{q}}\cdot \bm n, \varepsilon_{\widehat{y}}\rangle_{\partial\mathcal{T}_h}
        + \dfrac{1}{2}\langle\bm\beta\cdot\bm n \varepsilon_{\widehat{y}}, \varepsilon_{\widehat{y}}\rangle_{\partial\mathcal{T}_h}
        - \alpha (\varepsilon_z, \varepsilon_y)_{\mathcal{T}_h}\\
        &\qquad= - (I_{\bm q}, \bm \varepsilon_{\bm q})_{\mathcal{T}_h}
        + \alpha (I_z, \varepsilon_y)_{\mathcal{T}_h},\\
        &\norm[0]{\bm \varepsilon_{\bm p}}_{\mathcal{T}_h}^2
        + \norm[0]{(\tau_2 + \dfrac{1}{2} \bm \beta\cdot\bm n)^{1/2}(\varepsilon_z 
        - \varepsilon_{\widehat{z}})}_{\partial\mathcal{T}_h}^2
        + \langle\bm \varepsilon_{\widehat{p}}\cdot \bm n, \varepsilon_{\widehat{z}}\rangle_{\partial\mathcal{T}_h}
        - \dfrac{1}{2} \langle \bm \beta \cdot \bm n \varepsilon_{\widehat{z}}, \varepsilon_{\widehat{z}}\rangle_{\partial\mathcal{T}_h}
        + (\varepsilon_y, \varepsilon_z)_{\mathcal{T}_h}\\
        &\qquad= - (I_{\bm p}, \bm \varepsilon_{\bm p})_{\mathcal{T}_h}
        - (I_z, \varepsilon_z)_{\mathcal{T}_h}.
    \end{align*}
    Now, multiplying by $\gamma$ the first of the above equations and adding it to the second one, 
    \begin{align}
    \nonumber
        &\mathcal{E}
        + \gamma \langle\bm\varepsilon_{\widehat{\bm q}}\cdot\bm n 
        + \bm \beta \cdot \bm n \varepsilon_{\widehat{y}}, \varepsilon_{\widehat{y}}\rangle_{\Gamma_h}
        - \dfrac{\gamma}{2} \langle \bm \beta \cdot \bm n \varepsilon_{\widehat{y}}, \varepsilon_{\widehat{y}}\rangle_{\partial\mathcal{T}_h}
        + \langle \bm \varepsilon_{\widehat{\bm p}} \cdot \bm n 
        - \bm \beta \cdot \bm n \varepsilon_{\widehat{z}}, \varepsilon_{\widehat{z}}\rangle_{\Gamma_h}
        + \dfrac{1}{2} \langle \bm \beta \cdot \bm n \varepsilon_{\widehat{z}}, \varepsilon_{\widehat{z}}\rangle_{\partial\mathcal{T}_h}\\
        \label{eq:aux1-energy_identity}
        &\qquad= - \gamma (I_{\bm q}, \bm \varepsilon_{\bm q})_{\mathcal{T}_h}
        + (I_z, \varepsilon_y)_{\mathcal{T}_h}
        - (I_{\bm p}, \bm \varepsilon_{\bm p})_{\mathcal{T}_h}
        - (I_y, \varepsilon_z)_{\mathcal{T}_h},
    \end{align}
    where the second and fourth terms were obtained by \eqref{eq:energy_arg_e} and \eqref{eq:energy_arg_f}, respectively.
    Next, taking into account \eqref{eq:energy_arg_i} and \eqref{eq:energy_arg_j}, we can respectively rewrite the second and fourth terms of \eqref{eq:aux1-energy_identity} as:
    \begin{align}
        \label{eq:aux2-energy_identity}
        &\gamma \langle \bm \varepsilon_{\widehat{\bm q}} \cdot \bm n
        + \bm \beta \cdot \bm n \varepsilon_{\widehat{y}}, \varphi_1 
        - \varphi_1^h\rangle_{\Gamma_h}
        = \gamma \langle \bm \varepsilon_{\bm q} \cdot \bm n
        + \tau_1 (\varepsilon_y
        - \varepsilon_{\widehat{y}})
        + \bm \beta \cdot \bm n \varepsilon_{\widehat{y}}, \varphi_1
        - \varphi_1^h \rangle_{\Gamma_h},\\
        \label{eq:aux3-energy_identity}
        &\langle \bm \varepsilon_{\widehat{\bm p}} \cdot \bm n
        - \bm \beta \cdot \bm n \varepsilon_{\widehat{z}}, \varphi_2 
        - \varphi_2^h \rangle_{\Gamma_h}
        = \langle \bm \varepsilon_{\bm p} \cdot \bm n
        + \tau_2 (\varepsilon_z - \varepsilon_{\widehat{z}})
        - \bm \beta \cdot \bm n \varepsilon_{\widehat{z}}, \varphi_2
        - \varphi_2^h\rangle_{\Gamma_h}.
    \end{align}
    In turn, since $\varepsilon_{\widehat{y}}$ and $\varepsilon_{\widehat{z}}$ are single-valued functions across the interfaces, and bearing in mind \eqref{eq:energy_arg_g} and \eqref{eq:energy_arg_h}, there hold 
    $\langle\frac{1}{2} \bm \beta \cdot \bm n \varepsilon_{\widehat{y}}, \varepsilon_{\widehat{y}}\rangle_{\partial\mathcal{T}_h}
    = \langle \frac{1}{2} \bm \beta \cdot \bm n \varepsilon_{\widehat{y}}, \varphi_1 - \varphi_1^h \rangle_{\Gamma_h}$ and
    $\langle\frac{1}{2} \bm \beta \cdot \bm n \varepsilon_{\widehat{z}}, \varepsilon_{\widehat{z}}\rangle_{\partial\mathcal{T}_h}
    = \langle \frac{1}{2} \bm \beta \cdot \bm n \varepsilon_{\widehat{z}}, \varphi_2 - \varphi_2^h\rangle_{\Gamma_h}$. Then, combining the above with \eqref{eq:aux1-energy_identity}, \eqref{eq:aux2-energy_identity}, 
    and \eqref{eq:aux3-energy_identity}, after algebraic manipulations we get \eqref{lemma_eq:energy_identity}. 

    On the other hand, by adding and subtracting $\bm\Pi_{\bm V}q$, we find
    \begin{align*}
        \varphi_1 - \varphi_1^h 
        & = \int_0^{l(\bm x)}(\bm q - \bm E(\bm q_h))(\bm x(s)) \cdot \bm m \text{d}s 
        = \int_0^{l(\bm x)}(I_{\bm q} + \bm\varepsilon_{\bm q})(\bm x(s))\cdot \bm m \text{d}s \\
        & = l \big( - \Lambda^{\bm I_{\bm q}} 
        + \bm I_{\bm q} \cdot \bm m 
        - \Lambda^{\bm \varepsilon_{\bm q}}
        +\bm \varepsilon_{\bm q} \cdot \bm m )\big).
    \end{align*}

Now, decomposing $\bm m =(\bm m \cdot \bm n) \bm n+ \bm t$, we obtain
        \begin{align*}
        \varphi_1 - \varphi_1^h     
        & = l \big( - \Lambda^{\bm I_{\bm q}} 
        + \bm I_{\bm q} \cdot \bm m 
        - \Lambda^{\bm \varepsilon_{\bm q}}
        + (\bm \varepsilon_{\bm q} \cdot \bm n )(\bm m \cdot \bm n)
        + \bm \varepsilon_{\bm q} \cdot \bm t \big).
    \end{align*}

    After reordering the terms we deduce \eqref{eq:identity_ehq_ehp_a} and performing similar steps we deduce that \eqref{eq:identity_ehq_ehp_b} holds.
\end{proof}

\begin{corollary} 
\label{cor:energy_arg_ineq}
There exists a positive constant $C$, independent of $h$, such that
\begin{align}\label{eq:energy_arg_ineq}
    \mathcal{E} + \mathcal{E}_\varphi
    \leq C \Big\{ \mathbb{T}_{\bm q,y} + \mathbb{T}_{\bm p,z} + \Big(1 + \norm[0]{\boldsymbol{\beta}}_{L^{\infty}(\Omega)} R \Big)
    ( \norm[0]{\varepsilon_y}_{\mathcal{T}_h}^2 
    + \norm[0]{\varepsilon_z}_{\mathcal{T}_h}^2) \Big\},
\end{align}
where, for $(\bm r,v) \in \{(\bm q,y),(\bm p,z)\}$,
\begin{align}
    \mathbb{T}_{\bm r,v} 
    := \norm[0]{I_{\bm{r}}}_{\mathcal{T}_h}^2
    + \norm[0]{I_{v}}_{\mathcal{T}_h}^2
    + \max_{e \in \mathcal{E}_h} \beta_e^{-1} r_e \norm[0]{I_{\bm{r}} \cdot \bm m}_{\Gamma_h,h^{\perp}}^2
    + R^2 \norm[0]{\nabla I_{\bm{r}}}_{\Omega_h^c,(h^{\perp})^2}^2. \label{defTrv}
\end{align}
\end{corollary}
\begin{proof}
    By \eqref{eq:energy_arg_i} and \eqref{eq:identity_ehq_ehp_a}, we can rewrite $\mathbb K_y$ as follows
    \begin{align*}
        \mathbb{K}_y
        & := - \gamma \norm[0]{(\bm m \cdot \bm n)^{-1/2}(\varphi_1 - \varphi_1^h)}_{\Gamma_h,l^{-1}}^2
        - \gamma \langle (\bm m \cdot \bm n)^{-1} \Lambda^{I_{\bm q}}, \varphi_1 - \varphi_1^h \rangle_{\Gamma_h}
        - \gamma \langle (\bm m \cdot \bm n)^{-1} \Lambda^{\bm\varepsilon_{\bm q}}, \varphi_1 - \varphi_1^h \rangle_{\Gamma_h} \\
        & \quad + \gamma \langle (\bm m \cdot \bm n)^{-1} I_{\bm q} \cdot \bm m, \varphi_1 - \varphi_1^h \rangle_{\Gamma_h}
        + \gamma \langle (\bm m \cdot \bm n)^{-1} \bm \varepsilon_{\boldsymbol{q}} \cdot \bm t, \varphi_1 - \varphi_1^h \rangle_{\Gamma_h} \\
        & \quad - \gamma \langle(\tau_1 - \dfrac{1}{2} \bm\beta \cdot \bm n)(\varepsilon_y - \varepsilon_{\widehat{y}}), \varphi_1 - \varphi_1^h \rangle_{\Gamma_h}
        - \gamma \langle \dfrac{1}{2} \bm\beta \cdot \bm n \varepsilon_y, \varphi_1 - \varphi_1^h \rangle_{\Gamma_h}.
    \end{align*}
    Bearing in mind Assumption \ref{Assumption:2} and applying the discrete trace inequality, the Cauchy-Schwarz inequality, and Young's inequality to each term of $\mathbb{K}_y$, we obtain 
    \begin{align*}
        \mathbb{K}_y
        & \leq - \dfrac{1}{2} \norm[0]{(\bm m \cdot \bm n)^{-1/2} (\varphi_1 - \varphi_1^h)}_{\Gamma_h, l^{-1}}^2
        + 3 \max_{e \in \mathcal{E}_h} \beta_e^{-1} \gamma \norm[0]{l^{1/2} \Lambda^{I_{\bm q}}}_{\Gamma_h}^2
        + 3 \max_{e \in \mathcal{E}_h} \beta_e^{-1} \gamma \norm[0]{l^{1/2} \Lambda^{\bm \varepsilon_{\bm q}}}_{\Gamma_h}^2 \\
        & \quad + 3 \max_{e \in \mathcal{E}_h} \beta_e^{-1} r_e \gamma \norm[0]{I_{\bm q} \cdot \bm m}_{\Gamma_h, h^{\perp}}^2 \\
        & \quad + 3 \max_{e \in \mathcal{E}_h} \beta_e^{-1} \norm[0]{\bm t}_{L^{\infty}(\Gamma_h)}^2 C_{tr}^2 r_e \gamma \norm[0]{\bm \varepsilon_{\bm q}}_{\mathcal{T}_h}^2
        + \dfrac{3}{4}\norm[0]{(\tau_1 - \dfrac{1}{2}\bm \beta \cdot \bm n)^{1/2} (\varepsilon_y - \varepsilon_{\widehat{y}})}_{\Gamma_h}^2\\
        & \quad + \dfrac{3}{2} \norm[0]{\bm \beta}_{L^{\infty}(\Omega)}^2 \max_{e \in \mathcal{E}_h} r_e C_{tr}^2 \norm[0]{\varepsilon_y}_{\mathcal{T}_h}^2.
    \end{align*}
    The deduction of the last three terms follows from the facts that $\boldsymbol{\beta} \in [L^{\infty}(\Omega)]^n$, together with the discrete trace inequality with $C_{tr}$. Now, applying \eqref{ineq:Lambdas}, and considering that $R = \max_{e \in \mathcal{T}_h} r_e$, we find
    \begin{align}
        \nonumber
        \mathbb{K}_y
        & \leq - \dfrac{1}{2} \norm[0]{(\bm m \cdot \bm n)^{-1/2} (\varphi_1 - \varphi_1^h)}_{\Gamma_h, l^{-1}}^2
        + R^2 \gamma \norm[0]{\nabla I_{\bm p}}_{\Omega_h^c, (h^{\perp})^2}^2
        + \max_{e \in \mathcal{E}_h} \beta_e^{-1} r_e^3 C_{ext}^2 C_{inv}^2 \gamma \norm[0]{\bm \varepsilon_{\bm q}}_{\mathcal{T}_h}^2 \\
        \nonumber
        & \quad + 3 \max_{e \in \mathcal{E}_h} \beta_e r_e \gamma \norm[0]{I_{\bm q} \cdot \bm m}_{\Gamma_h, h^{\perp}}^2
        + 3 \max_{e \in \mathcal{E}_h} \beta_e r_e \norm[0]{\bm t}_{L^{\infty}(\Gamma_h)}^2 C_{tr}^2 \norm[0]{\bm \varepsilon_{\bm q}}_{\mathcal{T}_h}^2
        + \dfrac{3}{4} \norm[0]{(\tau_1 - \dfrac{1}{2} \bm \beta \cdot \bm n)^{1/2} (\varepsilon_y - \varepsilon_{\widehat{y}})}_{\Gamma_h}^2 \\
        \label{eq:energy_arg_ineq-aux1}
        & \quad + \dfrac{3}{2} \norm[0]{\bm \beta}_{L^{\infty}(\Omega)}^2 R C_{tr}^2 \norm[0]{\varepsilon_y}_{\mathcal{T}_h}^2 
    \end{align}
    An analogous estimate can be obtained for $\mathbb K_z$ following the same steps. On the other hand, applying Cauchy-Schwarz and Young's inequality in \eqref{lemma_eq:energy_identity} and combining with \eqref{eq:energy_arg_ineq-aux1}, we get
    \begin{align*}
        & \dfrac{\gamma}{2} \norm[0]{\bm \varepsilon_q}_{\mathcal{T}_h}^2
        + \gamma \norm[0]{(\tau_1 - \dfrac{1}{2} \bm \beta \cdot \bm n)^{1/2} (\varepsilon_y - \varepsilon_{\widehat{y}})}_{\partial \mathcal{T}_h}^2
        + \gamma \norm[0]{\varphi_1 - \varphi_1^h}_{\Gamma_h, l^{-1}}^2
        + \gamma \norm[0]{(\bm m \cdot \bm n)^{-1/2} (\varphi_1 - \varphi_1^h)}_{\Gamma_h, l^{-1}}^2 \\ 
        & + \dfrac{1}{2} \norm[0]{\bm \varepsilon_p}_{\mathcal{T}_h}^2
        + \norm[0]{(\tau_2 + \dfrac{1}{2} \bm \beta \cdot \bm n)^{1/2} (\varepsilon_z - \varepsilon_{\widehat{z}})}_{\partial \mathcal{T}_h}^2
        + \norm[0]{\varphi_2 - \varphi_2^h}_{\Gamma_h, l^{-1}}^2
        + \norm[0]{(\bm m \cdot \bm n)^{-1/2} (\varphi_2 - \varphi_2^h)}_{\Gamma_h, l^{-1}}^2 \\
        & \leq \dfrac{\gamma}{2} \norm[0]{I_{\bm q}}_{\mathcal{T}_h}^2
        + \dfrac{1}{2} \norm[0]{I_z}_{\mathcal{T}_h}^2
        + \dfrac{1}{2} \norm[0]{\varepsilon_y}_{\mathcal{T}_h}^2
        + \dfrac{1}{2} \norm[0]{I_{\bm p}}_{\mathcal{T}_h}^2 
        + \dfrac{1}{2} \norm[0]{I_y}_{\mathcal{T}_h}^2 
        + \dfrac{1}{2} \norm[0]{\varepsilon_z}_{\mathcal{T}_h}^2 \\
        & \quad + R^2 \gamma \norm[0]{\nabla I_{\bm q}}_{\Omega_h^c, (h^{\perp})^2}^2
        + \max_{e \in \mathcal{E}_h} \beta_e^{-2} r_e^3 C_{ext}^2 C_{inv}^2 \gamma \norm[0]{\bm \varepsilon_q}_{\mathcal{T}_h}^2
        + 3 \max_{e \in \mathcal{E}_h} \beta_e^{-1} r_e \gamma \norm[0]{I_{\bm q} \cdot \bm m}_{\Gamma_h, h^{\perp}}^2 \\
        & \quad + 3 \max_{e \in \mathcal{E}_h} \beta_e^{-1} r_e \gamma  \norm[0]{\bm t}_{L^{\infty}(\Gamma_h)} C_{tr}^2 \norm[0]{\bm \varepsilon_{\bm q}}_{\mathcal{T}_h}^2
        + \dfrac{3}{4} \gamma \norm[0]{(\tau_1 - \dfrac{1}{2} \bm \beta \cdot \bm n)^{1/2} (\varepsilon_y - \varepsilon_{\widehat{y}})}_{\Gamma_h}^2
        + \dfrac{3}{2} \gamma \norm[0]{\bm \beta}_{L^{\infty}(\Omega)} R C_{tr}^2 \norm[0]{\varepsilon_y}_{\mathcal{T}_h}^2 \\
        & \quad + R^2 \norm[0]{\nabla I_{\bm p}}_{\Omega_h^c, (h^{\perp})^2}^2
        + \max_{e \in \mathcal{E}_h} \beta_e^{-2} r_e^3 C_{ext}^2 C_{inv}^2 \norm[0]{\bm \varepsilon_p}_{\mathcal{T}_h}^2
        + 3 \max_{e \in \mathcal{E}_h} \beta_e^{-1} r_e \norm[0]{I_{\bm p} \cdot \bm m}_{\Gamma_h, h^{\perp}}^2 \\
        & \quad + 3 \max_{e \in \mathcal{E}_h} \beta_e^{-1} r_e \norm[0]{\bm t}_{L^{\infty}(\Gamma_h)}^2 C_{tr}^2 \norm[0]{\bm \varepsilon_{\bm p}}_{\mathcal{T}_h}^2
        + \dfrac{3}{4} \norm[0]{(\tau_1 - \dfrac{1}{2} \bm \beta \cdot \bm n)^{1/2} (\varepsilon_z - \varepsilon_{\widehat{z}})}_{\Gamma_h}^2
        + \dfrac{3}{2} \norm[0]{\bm \beta}_{L^{\infty}(\Omega)}^2 R C_{tr}^2 \norm[0]{\varepsilon_z}_{\mathcal{T}_h}^2.
     \end{align*}
     Therefore, using the closeness assumptions \cref{Assumption:1,Assumption:t} and reorganizing terms we find that \eqref{eq:energy_arg_ineq} holds.
\end{proof}

The estimate obtained in Corollary \ref{cor:energy_arg_ineq} allows us to 
control $\bm \varepsilon_q$ and $\bm \varepsilon_p$ as a function of 
$\varepsilon_y$ and $\varepsilon_z$. The latter terms will be bounded by 
making use of a duality argument, which is addressed in the next section.

\subsection{Duality argument}\label{ch:duality_argument}
In order to obtain the desire estimates for $\varepsilon_y$ and 
$\varepsilon_z$, we introduce the following equivalent optimal control 
problem (for more details, see \cite[Section 4.3]{ZHU20162}): 
\begin{equation*}
    \min_{\psi\in L^2(\Omega)} \dfrac{1}{2} \norm[0]{\Psi - \Theta_2}^2_{L^2(\Omega)} + \dfrac{\gamma}{2} \norm[0]{\psi}_{L^2(\Omega)}^2
\end{equation*}
subject to
\begin{subequations}
\begin{alignat*}{2}
    - \Delta \Psi - \bm\beta \cdot \nabla \Psi 
    & = \Theta_1 - \psi \qquad &&\text{in } \Omega, \\
    \Psi & = 0 \qquad && \text{on }\Gamma.
\end{alignat*}
\end{subequations}
From which the following dual system can be deduced:\\
\begin{subequations}\label{eq:dual_problem}
\begin{minipage}{0.45\textwidth}
\begin{alignat}{4}
    \label{eq:dual_problem_a}
    \bm \phi - \nabla \Psi 
    & = & \bm 0 & \quad\text{in }\Omega,\\
    \label{eq:dual_problem_b}
    - \nabla \cdot \bm\phi - \bm\beta \cdot \nabla\Psi + \psi  
    & = & \Theta_1 &\quad\text{in }\Omega,\\
    \label{eq:dual_problem_c}
    \Psi & = & 0 & \quad\text{on }\Gamma,
\end{alignat}
\end{minipage}   
\begin{minipage}{0.45\textwidth}
\begin{alignat}{4}    
    \label{eq:dual_problem_d}
    \bm \phi^{\bm p} - \nabla\Psi^{z} 
    & = & \bm 0 &\quad\text{in }\Omega,\\
    \label{eq:dual_problem_e}
    - \nabla \cdot \bm\phi^{\bm p} + \bm\beta \cdot \nabla\Psi^{z} - \Psi 
    & = & \Theta_2 &\quad \text{in }\Omega,\\
    \label{eq:dual_problem_f}
    \Psi^{z} & = & 0 &\quad \text{on }\Gamma,
\end{alignat}
\end{minipage}   

\noindent and
\begin{equation}\label{eq:dual_problem_g}
  \gamma \psi - \Psi^z = 0 \quad\text{in }\Omega.
\end{equation}
\end{subequations}
To shorten notation, let $\Theta:= \norm[0]{\Theta_1}_{\Omega} + 
\norm[0]{\Theta_2}_{\Omega}$. We assume the following regularity estimate (see \cite[Proposition 5.1]{manzonioptimal})
\begin{align}
    \label{eq:eliptic_regularity}
    \norm[0]{\Psi}_{2,\Omega} + \norm[0]{\bm\phi}_{1,\Omega}
    + \norm[0]{\Psi^z}_{2,\Omega} + \norm[0]{\bm\phi^{\bm r}}_{1,\Omega}
    & \lesssim \Theta.
\end{align}
This holds true, for instance, for convex
polyhedral domains or when $\Gamma$ is $\mathcal{C}^2$. On the other hand, we 
observe that \eqref{eq:dual_problem} is posed on $\Omega$, whereas the HDG 
method seeks the solution in $\Omega_h$. In other words, the duality argument 
will involve expressions to take into account the influence of the mismatch 
between $\Omega_h$ and $\Omega$ in \eqref{ch:duality_argument}. More 
precisely, we consider the following result that can be obtained similarly to 
the proof of \cite[Lemma 5.5]{paper2}.

\begin{lemma}
\label{lemma:dual_estimates}
Under the assumptions given in Section {\rm\ref{ch:preliminaries}} and 
assuming that {\rm\eqref{eq:eliptic_regularity}} holds true, we have 
\begin{align*}
    \norm[0]{\Psi - P_M\Psi}_{\Gamma_h,(h^{\perp})^{-1}}
    & \lesssim h \Theta,&
    \norm[0]{\Psi + l \partial_{\bm n}\Psi}_{\Gamma_h,l^{-3}}
    & \lesssim \Theta,\\
    \norm[0]{\partial_{\bm n}\Psi - P_M\partial_{\bm n}\Psi}_{\Gamma_h,l}
    & \lesssim R h \Theta ,&
    \norm[0]{\Psi}_{\Gamma_h,l^{-2}}
    & \lesssim \Theta.
\end{align*}
In turn, the same estimates hold for $\Psi^z$.
\end{lemma}

\begin{lemma}
\label{lemma:norm_Ehy_Ehz}
We have that
\begin{subequations}
\begin{align}
    \label{eq:norm_Ehy_Ehz_a}
    \begin{split}
    \norm[0]{\varepsilon_y}_{\mathcal{T}_h}^2
    = & - (I_{\bm q}, \bm \Pi_{\bm V} \bm \phi)_{\mathcal{T}_h}
    + (\bm\varepsilon_{\bm q}, \bm\phi - \bm\Pi_{\bm V}\bm\phi)_{\mathcal{T}_h}
    + (\alpha e^{z}, \Pi_W\Psi)_{\mathcal{T}_h}
    + (\varepsilon_y, \alpha \Psi^z)_{\mathcal{T}_h}
    + \mathbb{S}^y,\end{split}\\
    \label{eq:norm_Ehy_Ehz_b}
    \begin{split}
    \norm[0]{\varepsilon_z}_{\mathcal{T}_h}^2
    = & - (I_{\bm p}, \widetilde{\bm \Pi}_{\bm V} \bm\phi^{\bm p})_{\mathcal{T}_h}
    + (\bm \varepsilon_{\bm p}, \bm\phi^{\bm p} - \widetilde{\bm\Pi}_{\bm V} \bm\phi^{\bm p})_{\mathcal{T}_h}
    - (e^{y}, \widetilde{\Pi}_W \Psi^z)_{\mathcal{T}_h}
    - (\varepsilon_z, \Psi)_{\mathcal{T}_h}
    + \mathbb{S}^z,
    \end{split}
\end{align}
\end{subequations}
where,
\begin{align*}
    \mathbb{S}_y 
    &= - \langle \varepsilon_{\widehat y} - \varepsilon_y, \bm\beta \cdot \bm n (\Pi_W \Psi - P_M\Psi)\rangle_{\partial\mathcal{T}_h}
    - \langle \varepsilon_{\widehat{y}}, \bm\phi \cdot \bm n \rangle_{\Gamma_h}
    - \langle \bm\varepsilon_{\widehat{\bm q}} \cdot \bm n + \bm\beta \cdot \bm n \varepsilon_{\widehat y}, \Psi\rangle_{\Gamma_h},\\
    \mathbb{S}_z
    &= \langle \varepsilon_{\widehat{z}} - \varepsilon_{z}, \bm\beta \cdot \bm n (\widetilde{\Pi}_W \Psi^z - P_M\Psi^z) \rangle_{\partial \mathcal{T}_h}
    - \langle \varepsilon_{\widehat{z}}, \bm \phi^{\bm p} \cdot \bm n \rangle_{\Gamma_h}
    - \langle \bm\varepsilon_{\widehat{\bm p}} \cdot \bm n -\bm\beta \cdot \bm n \varepsilon_{\widehat z}, \Psi^z \rangle_{\Gamma_h}.
\end{align*}    
\end{lemma}
\begin{proof}

We follow the steps carried out in \cite[Lemma 4.3]{henriquez2023control} 
adapted to our context (see also \cite[Lemma 4.6]{ZHU20162}). Let $\Theta_1 
= \varepsilon_y$ in $\Omega_h$ and $\Theta_1 = 0$ in $\Omega\setminus 
\Omega_h$. Then, adding a subtracting $\bm\Pi_{\bm V}\bm \phi$ and 
$\Pi_{W} \Psi$, we have that
\begin{align*}
    \norm[0]{\varepsilon_y}_{\mathcal{T}_h}^2
    = & - (\varepsilon_y, \nabla \cdot \bm\Pi_{\bm V} \bm\phi)_{\mathcal{T}_h}
    - (\varepsilon_y, \nabla\cdot(\bm\phi-\bm\Pi_{\bm V}\bm \phi))_{\mathcal{T}_h}
    - (\bm \beta \varepsilon_y, \nabla \Pi_W \Psi)_{\mathcal{T}_h}
    - (\varepsilon_y, \bm\beta \cdot \nabla(\Psi - \Pi_W\Psi))_{\mathcal{T}_h}\\
    &
    + (\bm\varepsilon_{\bm q}, \bm \Pi_{\bm V}\bm \phi)_{\mathcal{T}_h}
    + (\bm\varepsilon_{\bm q}, \bm\phi - \bm\Pi_{\bm V}\bm\phi)_{\mathcal{T}_h}
    - (\bm\varepsilon_{\bm q}, \nabla \Pi_W \Psi)_{\mathcal{T}_h}
    - (\bm\varepsilon_{\bm q}, \nabla(\Psi - \Pi_{W} \Psi))_{\mathcal{T}_h} 
    + (\varepsilon_y, \alpha \Psi^z)_{\mathcal{T}_h}.
\end{align*}
Now, choosing $\bm r_1 = \bm\Pi_{\bm V} \bm \phi$ in \eqref{eq:energy_arg_a} 
and $w_1 = \Pi_W \Psi$ in \eqref{eq:energy_arg_b}, followed by simple 
algebraic arrangements, we deduce that
\begin{align}
\begin{split}
    \norm[0]{\varepsilon_y}_{\mathcal{T}_h}^2 
    = & - ( I_{\bm q}, \bm \Pi_{\bm V} \bm\phi)_{\mathcal{T}_h}
    - \langle \varepsilon_{\widehat{y}}, \bm\Pi_{\bm V} \bm\phi \cdot \bm n \rangle_{\partial\mathcal{T}_h}
    - \langle \bm\varepsilon_{\widehat{\bm q}} \cdot \bm n + \bm\beta \cdot \bm n \varepsilon_{\widehat y}, \Pi_W \Psi \rangle_{\partial\mathcal{T}_h}
    + (\alpha \varepsilon_z, \Pi_W \Psi)_{\mathcal{T}_h}\\
    & + (\alpha I_z, \Pi_W \Psi)_{\mathcal{T}_h}
    - (\varepsilon_y, \nabla \cdot (\bm\phi - \bm \Pi_{\bm V}\bm\phi))_{\mathcal{T}_h}
    - (\varepsilon_y, \bm\beta \cdot \nabla(\Psi - \Pi_W\Psi))_{\mathcal{T}_h}
    + (\bm\varepsilon_{\bm q}, \bm\phi - \bm\Pi_{\bm V}\bm\phi)_{\mathcal{T}_h}\\
    & - (\bm\varepsilon_{\bm q}, \nabla(\Psi - \Pi_W\Psi))_{\mathcal{T}_h}
    + (\varepsilon_y, \alpha \Psi^z)_{\mathcal{T}_h}\\
    = &-  (I_{\bm q}, \bm\Pi_{\bm V}\bm\phi)_{\mathcal{T}_h}
    + (\alpha e^z, \Pi_W \Psi)_{\mathcal{T}_h}
    + (\bm\varepsilon_{\bm q}, \bm\phi - \bm\Pi_{\bm V}\bm\phi)_{\mathcal{T}_h}
    + (\varepsilon_y, \alpha \Psi^z)_{\mathcal{T}_h}
    + \mathbb{S}_y,
    \end{split}
    \label{aux3}
\end{align}
where, for notation purposes only,
\begin{align*}
    \mathbb{S}_y
    := & - \langle \varepsilon_{\widehat{y}}, \bm\Pi_{\bm V} \bm\phi \cdot \bm n \rangle_{\partial \mathcal{T}_h}
    - \langle \bm\varepsilon_{\widehat{\bm q}} \cdot \bm n + \bm\beta \cdot \bm n \varepsilon_{\widehat y}, \Pi_W \Psi \rangle_{\partial \mathcal{T}_h}
    - (\varepsilon_y, \nabla \cdot (\bm\phi - \bm\Pi_{\bm V}\bm\phi))_{\mathcal{T}_h}
    - (\bm\varepsilon_{\bm q}, \nabla(\Psi - \Pi_W\Psi))_{\mathcal{T}_h}\\
    & - (\bm\beta \varepsilon_y, \nabla(\Psi - \Pi_W\Psi)).
\end{align*}
On the other hand, integrating by parts the above expression and using the 
HDG projection, particularly \eqref{eq:proj_hdg_a} and \eqref{eq:proj_hdg_b}, 
we deduce that
\begin{align*}
    \mathbb{S}_y 
    = & - \langle \varepsilon_{\widehat{y}} - \varepsilon_{y},(\bm\Pi_{\bm V} \bm\phi - \bm\phi) \cdot \bm n \rangle_{\partial\mathcal{T}_h}
    - \langle(\bm\varepsilon_{\widehat{\bm q}} - \bm\varepsilon_{\bm q}) \cdot \bm n, \Pi_W \Psi - \Psi\rangle_{\partial\mathcal{T}_h}
    - \langle \bm\beta \cdot \bm n \varepsilon_{\widehat y}, \Pi_W \Psi - \Psi\rangle_{\partial \mathcal{T}_h}\\
    & - \langle \bm\beta \cdot \bm n \varepsilon_y, \Psi - \Pi_W \Psi\rangle_{\partial\mathcal{T}_h}
    - \langle \varepsilon_{\widehat{y}}, \bm\phi \cdot \bm n \rangle_{\Gamma_h}
    - \langle \bm\varepsilon_{\widehat{\bm q}} \cdot \bm n + \bm\beta \cdot \bm n \varepsilon_{\widehat y}, \Psi\rangle_{\Gamma_h},
\end{align*}
where for the last two terms we have employed the fact that 
$\varepsilon_{\widehat{y}}$ is single-valued on $\mathcal{E}_h^{\circ}$ and 
\eqref{eq:energy_arg_e}. Moreover, by \eqref{eq:energy_arg_i} and suitably 
applying (\ref{eq:proj_hdg_c}), 
\begin{align*}
    \mathbb{S}_y 
    = & - \langle \varepsilon_{\widehat y} - \varepsilon_y, (\bm\Pi_{\bm V} \bm\phi - \bm\phi) \cdot \bm n + \bm\beta \cdot \bm n (P_M\Psi -\Psi) 
    + \tau_1 (\Psi - \Pi_W \Psi) \rangle_{\partial \mathcal{T}_h}
    - \langle \varepsilon_{\widehat y} - \varepsilon_y, \bm\beta \cdot \bm n (\Pi_W \Psi - P_M\Psi) \rangle_{\partial\mathcal{T}_h}\\
    &
    - \langle \varepsilon_{\widehat{y}}, \bm\phi \cdot \bm n \rangle_{\Gamma_h}
    - \langle \bm\varepsilon_{\widehat{\bm q}} \cdot \bm n + \bm\beta \cdot \bm n \varepsilon_{\widehat y}, \Psi \rangle_{\Gamma_h}\\
    = & - \langle \varepsilon_{\widehat y} - \varepsilon_y, \bm\beta \cdot \bm n (\Pi_W \Psi - P_M \Psi) \rangle_{\partial \mathcal{T}_h}
    - \langle \varepsilon_{\widehat{y}}, \bm\phi \cdot \bm n\rangle_{\Gamma_h}
    - \langle \bm\varepsilon_{\widehat{\bm q}} \cdot \bm n + \bm \beta \cdot \bm n \varepsilon_{\widehat y}, \Psi \rangle_{\Gamma_h},
\end{align*}
which together with \eqref{aux3} implies \eqref{eq:norm_Ehy_Ehz_a}.

Finally, by taking $\Theta_2 = \varepsilon_z$ in $\Omega_h$ and 
$\Theta_2 = 0$ in $\Omega \setminus \Omega_h$, setting $\bm r_2 
= \bm\Pi_{\bm V}\bm\phi^{\bm p}$ in (\ref{eq:energy_arg_c}) and $w_2 
= \Pi_W \Psi^z$ in (\ref{eq:energy_arg_d}) in (\ref{eq:dual_problem_e}), 
similar arguments yield to \eqref{eq:norm_Ehy_Ehz_b}. 
\end{proof}

The presence of the terms $- \langle\varepsilon_{\widehat y} - 
\varepsilon_y, \bm\beta \cdot \bm n (\Pi_W \Psi -
P_M\Psi) \rangle_{\partial \mathcal{T}_h}$ in $\mathbb{S}^y$ and 
$\langle \varepsilon_{\widehat{z}} - \varepsilon_{z}, \bm\beta \cdot
\bm n (\widetilde{\Pi}_W \Psi^z - P_M\Psi^z) \rangle_{\partial 
\mathcal{T}_h}$ in $\mathbb{S}^z$ are due the convection part of the 
differential equation. The remaining terms in $\mathbb{S}_y$ and 
$\mathbb{S}_z$ are due to the fact that $\Gamma_h$ does not fit $\Gamma$, 
otherwise both would vanish by \eqref{eq:dual_problem_c} and 
\eqref{eq:dual_problem_f}. Then, in the following lemma we rewrite 
$\mathbb{S}_y$ and $\mathbb{S}_z$ in order to quantify explicitly the 
influence of the mismatch between $\Gamma$ and $\Gamma_h$.

\begin{lemma}
\label{lemma:identity_Sy_Sz}

We have the following decomposition: $\mathbb{S}_y = \sum_{i=1}^{10} 
\mathbb{S}^i_y$ and $\mathbb{S}_z = \sum_{i=1}^{10} \mathbb{S}^i_z$, where

\begin{minipage}{0.45\textwidth}
\begin{align*}
    \mathbb{S}^1_{y} 
    & = - \langle \varepsilon_{\widehat y} - \varepsilon_y, \bm\beta \cdot \bm n (\Pi_W\Psi - \Psi) \rangle_{\partial \mathcal{T}_h},\\
    \mathbb{S}^2_{y} 
    & = - \langle \varepsilon_{\widehat y} - \varepsilon_y, \bm\beta \cdot \bm n (\Psi - P_M \Psi) \rangle_{\partial \mathcal{T}_h},\\
    \mathbb{S}^3_y 
    & = - \langle (\bm m \cdot \bm n l)^{-1} (\varphi_1 - \varphi_1^h), \Psi + \bm m \cdot \bm n l \partial_{\bm n} \Psi \rangle_{\Gamma_h} ,\\
    \mathbb{S}^4_y 
    & = - \langle \varphi_1 - \varphi_1^h, P_M \partial_{\bm n} \Psi - \partial_{\bm n} \Psi\rangle_{\Gamma_h},\\
    \mathbb{S}^5_y
    & = - \langle (\bm m \cdot \bm n)^{-1} \Lambda^{I_{\bm q}}, \Psi\rangle_{\Gamma_h},\\
    \mathbb{S}^6_y
    & = - \langle (\bm m \cdot \bm n)^{-1} \Lambda^{\bm\varepsilon_{\bm q}}, \Psi \rangle_{\Gamma_h},\\
    \mathbb{S}^7_y 
    & = - \langle (\bm m \cdot \bm n)^{-1} \tau_1 (\varepsilon_y - \varepsilon_{\widehat y}), \Psi\rangle_{\Gamma_h},\\
    \mathbb{S}^8_y
    & = - \langle \bm\beta \cdot \bm n (\varphi_1 - \varphi_1^h), \Psi \rangle_{\Gamma_h},\\
    \mathbb{S}^9_y
    & = \langle (\bm m \cdot \bm n)^{-1} I_{\bm q} \cdot \bm m, \Psi - P_M \Psi \rangle_{\Gamma_h},\\
    \mathbb{S}^{10}_y
    & = - \langle (\bm m \cdot \bm n)^{-1} \tau_1 I_y, P_M \Psi \rangle_{\Gamma_h}, \\
    \mathbb{S}^{11}_y
    & = - \langle (\bm m \cdot \bm n)^{-1} \bm \varepsilon_{\bm q} \cdot \bm t, \Psi \rangle_{\Gamma_h}, 
\end{align*}
\end{minipage}
 \begin{minipage}{0.45\textwidth}
\begin{align*}
    \mathbb{S}^1_{z} 
    & = \langle\varepsilon_{\widehat z} - \varepsilon_z, \bm\beta \cdot \bm n (\widetilde\Pi_W \Psi^z - \Psi^z) \rangle_{\partial\mathcal{T}_h},\\
    \mathbb{S}^2_{z} 
    & = \langle \varepsilon_{\widehat z} - \varepsilon_z, \bm\beta \cdot \bm n (\Psi^z - P_M \Psi^z) \rangle_{\partial \mathcal{T}_h},\\
    \mathbb{S}^3_z 
    & = - \langle (\bm m \cdot \bm n l)^{-1} (\varphi_2 - \varphi_2^h), \Psi^z + \bm m \cdot \bm n l \partial_{\bm n} \Psi^z\rangle_{\Gamma_h},\\
    \mathbb{S}^4_z 
    & = - \langle \varphi_2 - \varphi_2^h, P_M \partial_{\bm n} \Psi^z - \partial_{\bm n} \Psi^z \rangle_{\Gamma_h},\\
    \mathbb{S}^5_z
    & = - \langle (\bm m \cdot \bm n)^{-1} \Lambda^{I_{\bm p}}, \Psi^z \rangle_{\Gamma_h},\\
    \mathbb{S}^6_z
    & = - \langle (\bm m \cdot \bm n)^{-1} \Lambda^{\bm\varepsilon_{\bm p}}, \Psi^z \rangle_{\Gamma_h},\\
    \mathbb{S}^7_y
    & = - \langle (\bm m \cdot \bm n)^{-1} \tau_2 (\varepsilon_z - \varepsilon_{\widehat z}), \Psi^z \rangle_{\Gamma_h},\\
    \mathbb{S}^8_z
    & = \langle \bm\beta \cdot \bm n (\varphi_2 - \varphi_2^h), \Psi^z \rangle_{\Gamma_h},\\
    \mathbb{S}^9_z
    & = \langle (\bm m \cdot \bm n)^{-1} I_{\bm p} \cdot \bm m, \Psi^z - P_M \Psi^z \rangle_{\Gamma_h},\\
    \mathbb{S}^{10}_z
    & = - \langle (\bm m \cdot \bm n)^{-1} \tau_2 I_z, P_M \Psi^z \rangle_{\Gamma_h} \\
    \mathbb{S}^{11}_y
    & = - \langle (\bm m \cdot \bm n)^{-1} \bm \varepsilon_{\bm p} \cdot \bm t, \Psi^z \rangle_{\Gamma_h}, 
\end{align*}
\end{minipage} \\
\end{lemma}
\begin{proof}
We proceed analogously to the proof of \cite[Lemma 4.4]{henriquez2023control} 
(see also \cite[Lemma 5.4]{paper2}), but adding the treatment of the terms 
associated with the convection. Then, by \eqref{eq:energy_arg_i} and 
\eqref{eq:identity_ehq_ehp_a} we find
\begin{align*}
    \bm\varepsilon_{\widehat{\bm q}}\cdot\bm n
    & = \bm\varepsilon_{\bm q} \cdot \bm n + \tau_1 (\varepsilon_y - \varepsilon_{\widehat{y}})
    = (\bm m \cdot \bm n l)^{-1} (\varphi_1 - \varphi_1^h)
    + (\bm m \cdot \bm n)^{-1} \big( \Lambda^{I_{\bm q}}
    + \Lambda^{\bm\varepsilon_{\bm q}}
    - I_{\bm q}\cdot\bm n 
    - \bm \varepsilon_{\bm q} \cdot \bm t \big)
    + \tau_1 (\varepsilon_y - \varepsilon_{\widehat{y}}),
\end{align*}
then
\begin{align*}
    \mathbb{S}_y 
    = & - \langle \varepsilon_{\widehat{y}}, \bm\phi \cdot \bm n \rangle_{\Gamma_h} \\ 
    & - \langle (\bm m \cdot \bm n l)^{-1} (\varphi_1 - \varphi_1^h)
    + (\bm m \cdot \bm n)^{-1} \big( \Lambda^{I_{\bm q}}
    + \Lambda^{\bm\varepsilon_{\bm q}}
    - I_{\bm q} \cdot \bm n 
    - \bm \varepsilon_{\bm q} \cdot \bm t \big)
    + \tau_1 (\varepsilon_y - \varepsilon_{\widehat{y}})
    + \bm\beta \cdot \bm n \varepsilon_{\widehat y}, \Psi\rangle_{\Gamma_h}\\
    & - \langle \varepsilon_{\widehat y} - \varepsilon_y, \bm\beta \cdot \bm n (\Pi_W \Psi - P_M \Psi) \rangle_{\partial \mathcal{T}_h}.
\end{align*}
Now, bearing in mind that $\bm \phi = \nabla \Psi$ in \eqref{eq:dual_problem_a} 
and applying \eqref{eq:energy_arg_g} we deduce that 
\begin{align}\label{eq1:identity_Sy_Sz}
    -\langle \varepsilon_{\widehat y}, \bm \phi \cdot \bm n \rangle_{\Gamma_h}
    = - \langle \varphi_1 - \varphi_1^h, P_M \partial_{\bm n} \Psi \rangle_{\Gamma_h} \quad\text{and} \quad
    - \langle \bm \beta \cdot \bm n \varepsilon_{\widehat y}, \Psi \rangle_{\Gamma_h}
    = - \langle \bm \beta \cdot \bm n (\varphi_1 - \varphi_1^h), \Psi\rangle_{\Gamma_h}.
\end{align}
In turn, by \eqref{eq:proj_hdg_c}, we have that $\langle I_{\bm q} \cdot \bm n, \Psi\rangle_{\Gamma_h} = \langle I_{\bm q} \cdot \bm n, \Psi - P_M \Psi \rangle_{\Gamma_h} - \langle \tau_1 I_y, P_M \Psi \rangle_{\Gamma_h}$. Finally, by adding and subtracting $\Psi$ in the first argument of the last term of \eqref{eq1:identity_Sy_Sz}, and after a simple rearrangement of terms, we conclude the identity for $\mathbb S_y$. The decomposition for $\mathbb{S}_z$ can be obtained analogously. 
\end{proof}

\begin{corollary}
\label{lemma:ineq_Sy_Sz}
There holds that
\begin{align*}
    |\mathbb{S}_y|
    \lesssim & \Big(h^{3/2} + R^{1/2} h^{1/2} \Big)
    \left( \mathcal{E}_{\bm q, y} + \mathcal{E}_{\varphi_1}\right)^{1/2}
    \Theta + \widetilde{\mathbb T}_{\bm q,y}^{1/2} \Theta,
\end{align*}
where,
\begin{equation}\label{eq:T-tilde}
    \widetilde{\mathbb T}_{\bm q,y}
    := R^{3} h \norm[0]{\nabla I_{\bm q}}_{\Omega_h^c,(h^{\perp})^2}^2 
    + h^2 \norm[0]{I_{\bm q} \cdot \bm m}_{\Gamma_h, h^{\perp}}^2 
    + R^2 h \norm[0]{I_y }_{\Gamma_h, h^{\perp}}^2.
\end{equation}
A similar estimate holds true for $ |\mathbb{S}_z|$, where $(\bm 
p, z, \varphi_2)$ plays the role of $(\bm q, y, \varphi_1)$.
\end{corollary}
\begin{proof}
Lemmas \ref{lemma:dual_estimates} and \ref{lemma:identity_Sy_Sz} imply
\begin{align*}
    |\mathbb S_y^1| + |\mathbb S_y^2| + |\mathbb S_y^7|
    = &|\langle(\tau_1 - \dfrac{1}{2} \bm\beta \cdot \bm n)^{1/2}(\varepsilon_{y} - \varepsilon_{\widehat  y}), (\tau_1 - \dfrac{1}{2} \bm\beta \cdot \bm n)^{-1/2} \bm\beta \cdot \bm n (\Pi_W \Psi - \Psi) \rangle_{\partial \mathcal{T}_h}|\\
    & + |\langle(\tau_1 - \dfrac{1}{2} \bm\beta \cdot \bm n)^{1/2}(\varepsilon_{y} - \varepsilon_{\widehat y}), (\tau_1 - \dfrac{1}{2} \bm\beta \cdot \bm n)^{-1/2} \bm\beta \cdot \bm n (\Psi - P_M\Psi) \rangle_{\partial \mathcal{T}_h}|\\
    &
    + |\langle(\tau_1 - \dfrac{1}{2} \bm\beta \cdot \bm n)^{1/2}(\varepsilon_{y} - \varepsilon_{\widehat y}), (\tau_1 - \dfrac{1}{2} \bm\beta \cdot \bm n)^{-1/2} \tau_1 \Psi\rangle_{\Gamma_h}|.
\end{align*}
Applying the Cauchy-Schwarz and Young's inequalities, using the fact that $\bm\beta\in [W^{1,\infty}(\Omega)]^n$ and recalling \ref{Assumption:B-4}, we deduce that there exists $C_\tau > 0$ such that $(\tau_1 - \dfrac{1}{2} \bm \beta \cdot \bm n_e)^{-1/2} \leq C_\tau$ for all $e \in \mathcal{E}_h$. Thus, we have that
\begin{align*}
    |\mathbb S_y^1| + |\mathbb S_y^2| + |\mathbb S_y^7|
    \leq & \norm[0]{(\tau_1 - \dfrac{1}{2} \bm\beta \cdot \bm n)^{1/2}(\varepsilon_{y} - \varepsilon_{\widehat  y})}_{\partial \mathcal{T}_h} C_{\tau} \Big(\norm[0]{\bm\beta}_{L^{\infty}(\Omega)} \norm[0]{\Pi_W \Psi - \Psi}_{\partial \mathcal{T}_h}\\
    & + \norm[0]{\bm\beta}_{L^{\infty}(\Omega)} \norm[0]{\Psi - P_M\Psi}_{\partial \mathcal{T}_h} 
    + (\bm m \cdot \bm n)^{-1} \tau_1 R h \norm[0]{\Psi}_{\Gamma_h,l^{-2}}\Big).
\end{align*}
Now, bearing in mind that $\Psi \in H^2(\Omega)$ and applying \eqref{eq:hdg_proj_err_b}, \eqref{eq:proj-L2-faces-estimate}, Lemma \ref{lemma:dual_estimates} and the regularity estimate \eqref{eq:eliptic_regularity}, we obtain
\begin{equation*}
    |\mathbb S_y^1| + |\mathbb S_y^2| + |\mathbb S_y^7|
    \lesssim \norm[0]{(\tau_1 - \dfrac{1}{2} \bm\beta \cdot \bm n)^{1/2}(\varepsilon_{y} - \varepsilon_{\widehat  y})}_{\partial \mathcal{T}_h} \Big(h^{3/2} + R h\Big) \Theta,
\end{equation*}
where $R$ appears because, for each $e \in \Gamma_h$, $\displaystyle l(\bm x)\,\leq\, h_e^{\perp}\,r_e $. On the other hand, from the proof of \cite[Lemma 4.4]{henriquez2023control} and and considering that $(\bm m \cdot \bm n_e)^{-1} \leq \beta_e^{-1}$ for all $e \in \mathcal{E}_h$, we obtain the following estimates
\begin{align*}
    |\mathbb{S}_y^3|
    & \lesssim R^{1/2} h^{1/2} \norm[0]{\varphi_1 - \varphi_1^h}_{\Gamma_h, l^{-1}} \Theta, 
    & |\mathbb{S}_y^4|
    & \lesssim R h \norm[0]{\varphi_1 - \varphi_1^h}_{\Gamma_h,l^{-1}} \Theta, \\
    |\mathbb{S}_y^5|
    & \lesssim R^{3/2} h^{1/2} \norm[0]{\nabla I_{\bm q}}_{\Omega_h^c, (h^{\perp})^2} \Theta,
    &|\mathbb{S}_y^6|
    & \lesssim R^{2} h^{1/2} \norm[0]{\bm \varepsilon_{\bm q}}_{\mathcal{T}_h} \Theta, \\
    |\mathbb{S}_y^{8}|
    & \lesssim R^{3/2} h^{3/2} \norm[0]{\bm \beta}_{L^{\infty}(\Omega)} \norm[0]{\varphi_1 - \varphi_1^h}_{\Gamma_h, l^{-1}} \Theta,
    & |\mathbb{S}_y^{9}|
    & \lesssim h \norm[0]{I_{\bm q} \cdot \bm m}_{\Gamma_h, h^{\perp}} \Theta, \\
    |\mathbb{S}_y^{10}|
    & \lesssim R h^{1/2} \norm[0]{I_y}_{\Gamma_h, h^{\perp}} \Theta,
    & |\mathbb{S}_y^{11}|
    & \lesssim R h^{1/2} \norm[0]{\bm t}_{L^{\infty}(\Gamma_h)} \norm[0]{\bm \varepsilon_{\bm q}}_{\mathcal{T}_h} \Theta.
\end{align*}
 The result follows from the definition of $ \left(\mathcal{E}_{\bm p, y} +\mathcal{E}_{\varphi_1}\right)^{1/2}$ (cf. \eqref{def:mathcal_Epy_Epz} and \eqref{def:mathcal_Evarphi1,Evarphi2}). The estimate for $|\mathbb S_z|$ can be deduced following the same steps, but considering Assumption \ref{Assumption:B-2} and asking for $C_{\tau}$ to satisfy $(\tau_2 + \dfrac{1}{2} \bm\beta \cdot \bm n)^{-1/2} \leq C_\tau$.
\end{proof}

We are now in position to present the main results of our work. The first one controls the $L^2$-norm of the projection of the error associated to the scalar unknowns, whereas the second one provides the corresponding estimates for the mixed unknowns of the state and adjoint approximations.

\begin{theorem}
\label{lemma:EyEz_estimates}
There exists $h_0 > 0$, such that for all $h < h_0$, we have that
\begin{align}
\label{lemma:EyEz_estimatesA}
\begin{split}
    \norm[0]{\varepsilon_{y}}_{\mathcal{T}_h} + \norm[0]{\varepsilon_{z}}_{\mathcal{T}_h}
    \lesssim & (h + 1) (\norm[0]{I_{\bm q}}_{\mathcal{T}_h} 
    + \norm[0]{I_{\bm p}}_{\mathcal{T}_h})
    + \norm[0]{I_y}_{\mathcal{T}_h}
    + \norm[0]{I_z}_{\mathcal{T}_h}
    + \Big(h + R^2 h^{1/2} \Big) \Big(\mathbb T_{\bm q,y}^{1/2}
    + \mathbb T_{\bm p,z}^{1/2}\Big)\\
    &
    + \widetilde{\mathbb T}_{\bm q,y}^{1/2}
    + \widetilde{\mathbb T}_{\bm p,z}^{1/2}.
\end{split}
\end{align}
\end{theorem}
\begin{proof}
Adding \eqref{eq:norm_Ehy_Ehz_a} and \eqref{eq:norm_Ehy_Ehz_b}, we get
\begin{align*}
    \gamma \norm[0]{\varepsilon_y}_{\mathcal{T}_h}^2 + \norm[0]{\varepsilon_z}_{\mathcal{T}_h}^2
    = & - \gamma \left(I_{\bm q}, \bm\Pi_{\bm V} \bm\phi \right)_{\mathcal{T}_h}
    + \gamma \left(\bm\varepsilon_{\bm q}, \bm\phi - \bm\Pi_{\bm V} \bm\phi \right)_{\mathcal{T}_h}
    - (I_{\bm p}, \widetilde{\bm \Pi}_{\bm V} \bm\phi^{\bm p})_{\mathcal{T}_h}
    + (\bm \varepsilon_{\bm p}, \bm \phi^{\bm p} - \widetilde{\bm \Pi}_{\bm V} \bm \phi^{\bm p})_{\mathcal{T}_h}\\
    & + (e^{z}, \Pi_{W}\Psi)_{\mathcal{T}_h}
    + (\varepsilon_y, \Psi^z)_{\mathcal{T}_h}
    - (e^{y}, \widetilde{\Pi}_W \Psi^z)_{\mathcal{T}_h}
    - (\varepsilon_z, \Psi)_{\mathcal{T}_h}
    + \gamma \mathbb{S}^y
    + \mathbb{S}^z.
\end{align*}
In turn, applying the estimates of the HDG projections \eqref{eq:hdg_proj_err_a} and \eqref{eq:hdg_proj_err_d_a}, rearranging terms and bearing in mind the definition of $\mathcal E$ (cf. \eqref{def:mathcal_Epy_Epz}), we deduce
\begin{align*}
    - \gamma (I_{\bm q}, \bm\Pi_{\bm V} \bm\phi)_{\mathcal{T}_h}
    & \lesssim \norm[0]{I_{\bm q}}_{\mathcal{T}_h} (h + 1) \Theta,
    & - (I_{\bm p}, \widetilde{\bm\Pi}_{\bm V} \bm \phi^{\bm p})_{\mathcal{T}_h}
    & \lesssim \norm[0]{I_{\bm p}}_{\mathcal{T}_h} (h + 1) \Theta,\\ 
    \gamma (\bm\varepsilon_{\bm q}, \bm\phi - \bm\Pi_{\bm V} \bm\phi)_{\mathcal{T}_h}
    & \lesssim (\mathcal{E}_{\bm q,y} 
    + \mathcal{E}_{\varphi_1})^{1/2} h \Theta,
    & (\bm\varepsilon_{\bm p}, \bm\phi^{\bm p} - \widetilde{\bm \Pi}_{\bm V} \bm\phi^{\bm p})_{\mathcal{T}_h}
    & \lesssim (\mathcal{E}_{\bm p,z}
    + \mathcal{E}_{\varphi_2})^{1/2} h \Theta.
\end{align*}
In addition, by algebraic manipulations, \eqref{eq:hdg_proj_err_a}, \eqref{eq:hdg_proj_err_d_a}, and the regularity estimates \eqref{eq:eliptic_regularity}, we find
\begin{equation*}
    (e^z,\Pi_W\Psi)_{\mathcal{T}_h}
    + (\varepsilon_y, \Psi^z)_{\mathcal{T}_h}
    - (e^y, \widetilde{\Pi}_W\Psi^z)_{\mathcal{T}_h}
    - (\varepsilon_z, \Psi)_{\mathcal{T}_h} 
    \lesssim \Big(h \norm[0]{\varepsilon_y}_{\mathcal{T}_h}
    + \norm[0]{I_y}_{\mathcal{T}_h}
    + h \norm[0]{\varepsilon_z}_{\mathcal{T}_h} 
    + \norm[0]{I_z}_{\mathcal{T}_h}\Big) \Theta.
\end{equation*}
Now, setting $\Theta_1 = \varepsilon_y$ in $\Omega_h$, $\Theta_1 = 0$ in $\Omega \backslash \Omega_h$ in \eqref{eq:dual_problem_b} and $\Theta_2 = \varepsilon_z$ in $\Omega_h$, $\Theta_2 = 0$ in $\Omega \backslash \Omega_h$ in \eqref{eq:dual_problem_e} and applying Corollary \ref{lemma:ineq_Sy_Sz}, we arrive to 
\begin{align*}
    \norm[0]{\varepsilon_y}_{\mathcal{T}_h}
    + \norm[0]{\varepsilon_z}_{\mathcal{T}_h}
    \lesssim & (h + 1) (\norm[0]{I_{\bm q}}_{\mathcal{T}_h}
    + \norm[0]{I_{\bm p}}_{\mathcal{T}_h})
    + \norm[0]{I_{y}}_{\mathcal{T}_h}
    + \norm[0]{I_{z}}_{\mathcal{T}_h}
    + (h + R^{1/2} h^{1/2}) (\mathcal{E} + \mathcal{E}_{\varphi})^{1/2}
    + \Tilde{\mathbb{T}}_{q, y}^{1/2}
    + \Tilde{\mathbb{T}}_{p, z}^{1/2}.
\end{align*}
Then, by Corollary \ref{cor:energy_arg_ineq} and making algebraic rearrangements, we have that
\begin{align*}
    \norm[0]{\varepsilon_y}_{\mathcal{T}_h}
    + \norm[0]{\varepsilon_z}_{\mathcal{T}_h}
    \lesssim & (h + 1) (\norm[0]{I_{\bm q}}_{\mathcal{T}_h}
    + \norm[0]{I_{\bm p}}_{\mathcal{T}_h})
    + \norm[0]{I_y}_{\mathcal{T}_h}
    + \norm[0]{I_z}_{\mathcal{T}_h}
    + \Big(h + R^{1/2} h^{1/2}\Big) \Big(\mathbb T_{\bm q,y}^{1/2} 
    + \mathbb T_{\bm p,z}^{1/2}\Big)\\
    & + \Big(h + R^{1/2} h^{1/2}\Big)(\norm[0]{\varepsilon_y}_{\mathcal{T}_h} 
    + \norm[0]{\varepsilon_z}_{\mathcal{T}_h})
    + \widetilde{\mathbb T}_{\bm q,y}^{1/2}
    + \widetilde{\mathbb T}_{\bm p,z}^{1/2}.   
\end{align*}
Therefore, there exists $h_0 > 0$, such that, for all $h < h_0$ the estimate \eqref{lemma:EyEz_estimatesA} holds.
\end{proof}

\begin{corollary}\label{corollary:Ehv_bound_h^k} Let us assume $\bm q$ and $\bm p$ in $\bm H^{k+1}(\Omega)$; and $y$ and $z$ in $H^{k+1}(\Omega)$. Theremexists $h_0$ such that, for all   $h < h_0$, there holds
\begin{subequations}
\begin{align}  
    \norm[0]{y - y_h}_{\mathcal{T}_h}
    + \norm[0]{z - z_h}_{\mathcal{T}_h}
    & \lesssim h^{k+1} \left( |\bm q|_{k+1,\Omega} 
    + |\bm p|_{k+1,\Omega} 
    + |y|_{k+1,\Omega}
    + |z|_{k+1,\Omega}\right) \label{corA}
\end{align}
and
\begin{align}
    \norm[0]{\bm q - \bm q_h}_{\mathcal{T}_h} 
    + \norm[0]{\bm p - \bm p_h}_{\mathcal{T}_h}
    & \lesssim h^{k+1} \left(|\bm q|_{k+1,\Omega} 
    + |\bm p|_{k+1,\Omega} 
    + |y|_{k+1,\Omega}
    + |z|_{k+1,\Omega}\right).\label{corB}
\end{align}
\end{subequations}
\end{corollary}
\begin{proof}
First of all, we note that by a scaling argument and the properties of the HDG projection, we can show that 
\[\norm[0]{I_{W} y}_{\Gamma_h,h^{\perp}} + \norm[0]{I_W z}_{\Gamma_h,h^{\perp}} 
\lesssim  h^{k+1} \left(|y|_{k+1,\Omega} +|z|_{k+1,\Omega}\right).\]
Then, bearing in mind the definitions \eqref{defTrv} and \eqref{eq:T-tilde}, by the approximation properties of the HDG projections (cf. \eqref{eq:hdg_proj_err}, \eqref{eq:Ipn_Irn}), and the decomposition  $\bm m =(\bm m \cdot \bm n) \bm n+ \bm t$, we have that
\begin{align*}
    \mathbb{T}_{\bm q,y}^{1/2}
    + \widetilde{\mathbb{T}}_{\bm q,y}^{1/2}
    \lesssim h^{k+1} \left(|\bm q|_{k+1,\Omega}
    + | y|_{k+1,\Omega}\right)
\end{align*}
and a similar expression is obtained for $\mathbb{T}_{\bm p,z} $ and $\widetilde{\mathbb{T}}_{\bm p,z}$. Moreover, $\norm[0]{I_{y}}_{\mathcal{T}_h} + \norm[0]{I_z}_{\mathcal{T}_h} \lesssim  h^{k+1}\left(|y|_{k+1,\Omega} + |z|_{k+1,\Omega}\right)$. Thus, since $y - y_h = \varepsilon_y + I_y$ and $z - z_h = \varepsilon_z + I_z$, \eqref{corA} follows. Finally, \eqref{corB} can be deduced from Corollary \ref{cor:energy_arg_ineq} (cf. \eqref{eq:energy_arg_ineq}) by considering similar arguments.
\end{proof}

The estimates over the entire domain $\Omega$ can be obtained thanks to Corollary \ref{corollary:Ehv_bound_h^k} and \cite[Lemma 3.7]{paper2} , that is, 
\begin{equation*}
    \norm[0]{\bm q - \bm q_h}_{\Omega}
    + \norm[0]{\bm p - \bm p_h}_{\Omega} 
    + \norm[0]{y - y_h}_{\Omega}
    + \norm[0]{z - z_h}_{\Omega}
    \lesssim h^{k+1}.
\end{equation*}

We end this section by mentioning how to proof the well-posedness of the scheme (cf. Lemma \eqref{lemma:EU}).
\begin{proof}
Assume $f=0$ and $g=0$. By Fredholm alternative, it is enough to show that the solution to \eqref{equations:hdg} is $(\bm q_h, y_h, \widehat{y}_h, \bm p_h, z_h, \widehat{z}_h)=(\bm 0, 0, 0, \bm 0, 0, 0)$. To that end, since the left-hand side of  \eqref{equations:hdg}  is similar to that of \eqref{eq:energy_arg}, by proceeding exactly as in the proof of Corollary \ref{cor:energy_arg_ineq}, we can prove that there exists a positive constant $C$, independent of $h$, such that
\begin{align}
    E + E_\varphi
    \leq C \Big(1 + \norm[0]{\boldsymbol{\beta}}_{L^{\infty}(\Omega)} R \Big)
    ( \norm[0]{y_h}_{\mathcal{T}_h}^2 
    + \norm[0]{z_h}_{\mathcal{T}_h}^2),
\end{align}
where we recall the definition of $E$ and $E_\varphi$ in \eqref{def:mathcal_Epy_Epz_mathcal_Evarphi1,Evarphi2}. Now, by the same steps as in the proof of Theorem \eqref{lemma:EyEz_estimates}, we can show that there exists $h_0 > 0$, such that for all $h < h_0$, we have that $\norm[0]{y_h}_{\mathcal{T}_h} + \norm[0]{z_h}_{\mathcal{T}_h}   \lesssim 0$ and the result follows.
\end{proof}

\section{Numerical experiments}\label{ch:numerical_experiments}
In this section we present numerical experiments to validate the theoretical orders of convergence of the approximation provided by the HDG method in the two-dimensional case. For all the computations we consider the spaces specified in \eqref{eq:discrete_spaces} with $k\in \{0, 1, 2, 3\}$ and the exact solutions $y = \sin(\pi x)$ for the state and $z  = \sin(\pi x) \sin(\pi y)$ for the adjoint. We fix $\alpha = 1$, $\tau_1 = 1$ and $\tau_2$ according to assumption \ref{Assumption:B-2}. We present two different numerical examples:
\begin{enumerate}
    \item[] \label{ex1}{\bf Example 1:} We consider a circular domain  $\Omega :=  \{(x,y) \in \mathbb{R}^2 : x^2 + y^2 \leq 0.75\}$ and the vector field of the convective term as $\bm \beta = [1,1]$.
    \item[] {\bf Example 2:} We consider a kidney-shaped domain whose boundary satisfies the equation
    \[(2 [(x + 0.5)^2 + y^2] - x - 0.5)^2 - (x + 0.5)^2 + y^2] + 0.1 = 0.\]
    Additionally, we set the vector field of the convective term as $\bm \beta = [y, x]$.
\end{enumerate}

According to Corollary \ref{corollary:Ehv_bound_h^k}, the theoretical order of convergence for the $L^2$-norm of the errors in all the variables is $k + 1$, as long as Assumption \ref{Assumption:1} - \ref{Assumption:4} hold true. This is what we actually observe in the numerical experiments. More precisely, we have performed numerical simulations where $\Gamma_h$ is a piece-wise linear interpolation of  $\Gamma$ by a piece-wise. In this case, the distance between $\Gamma$ and $\Gamma_h$ is of order $h^2$, $R$ is proportional to $h$ and therefore the set of assumptions is valid for $h$ sufficiently small. The results (not reported here) showed the optimal convergence rate predicted by Corollary \ref{corollary:Ehv_bound_h^k}. 

On the other hand, we do report a more interesting and practical situation where the computational domain is constructed by ``embedding'' $\Omega$ in a background mesh and considering  $\Omega_h$ as the union of the elements lying completely inside of $\Omega$ as depicted in Figures \ref{fig:circle} and \ref{fig:kidney}. In this setting, $R$ is of order one and we cannot guaranty \ref{Assumption:1} holds. However, as we observe in Tables \ref{table:y,p,yhat}-\ref{table:z,r,zhat} (Example 1) and Tables \ref{table:y,p,yhat_kidney}-\ref{table:z,r,zhat_kidney} (Example 2), the order of convergence in all the variables is still $k+1$.

\begin{figure}[ht!]
\centering
\begin{minipage}{.5\textwidth}
  \centering
  \includegraphics[width=1\linewidth]{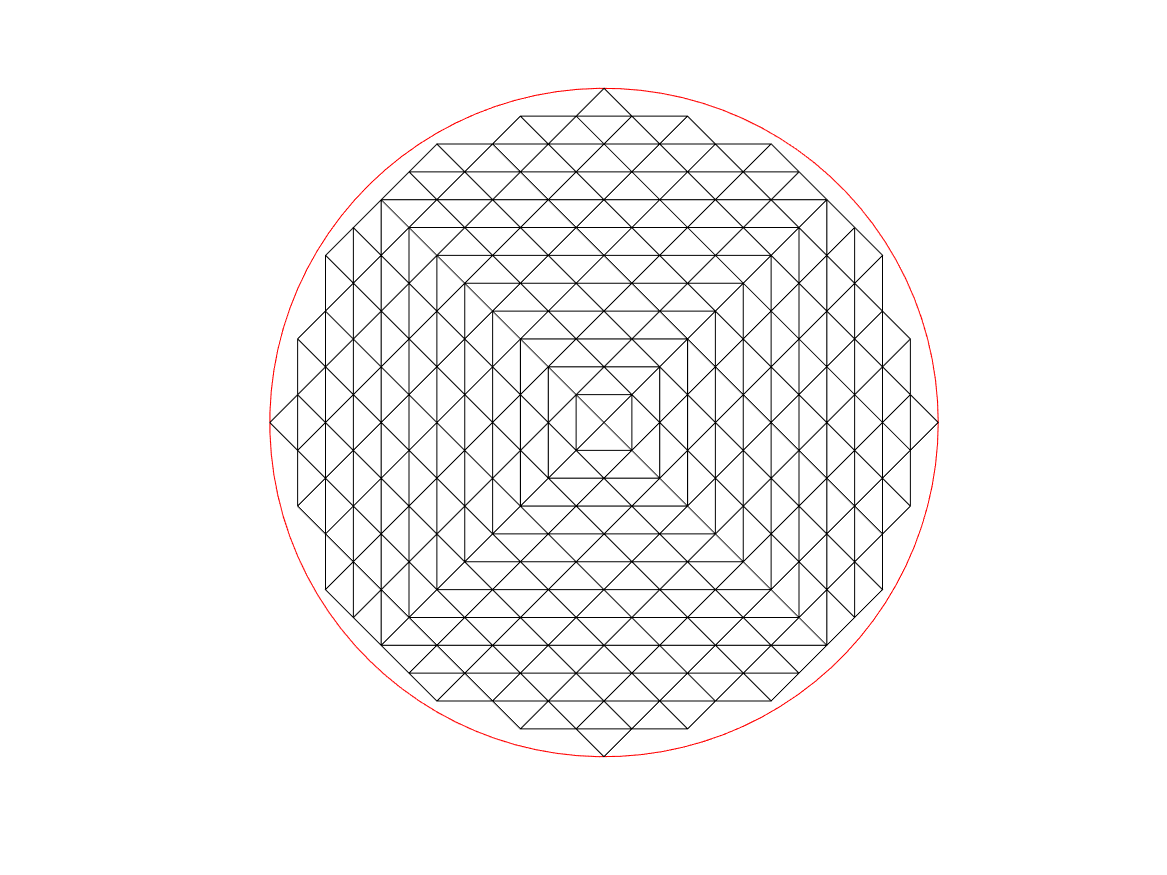}
  \captionof{figure}{Representation of a circle domain.}
  \label{fig:circle}
\end{minipage}%
\begin{minipage}{.5\textwidth}
  \centering
  \includegraphics[width=1\linewidth]{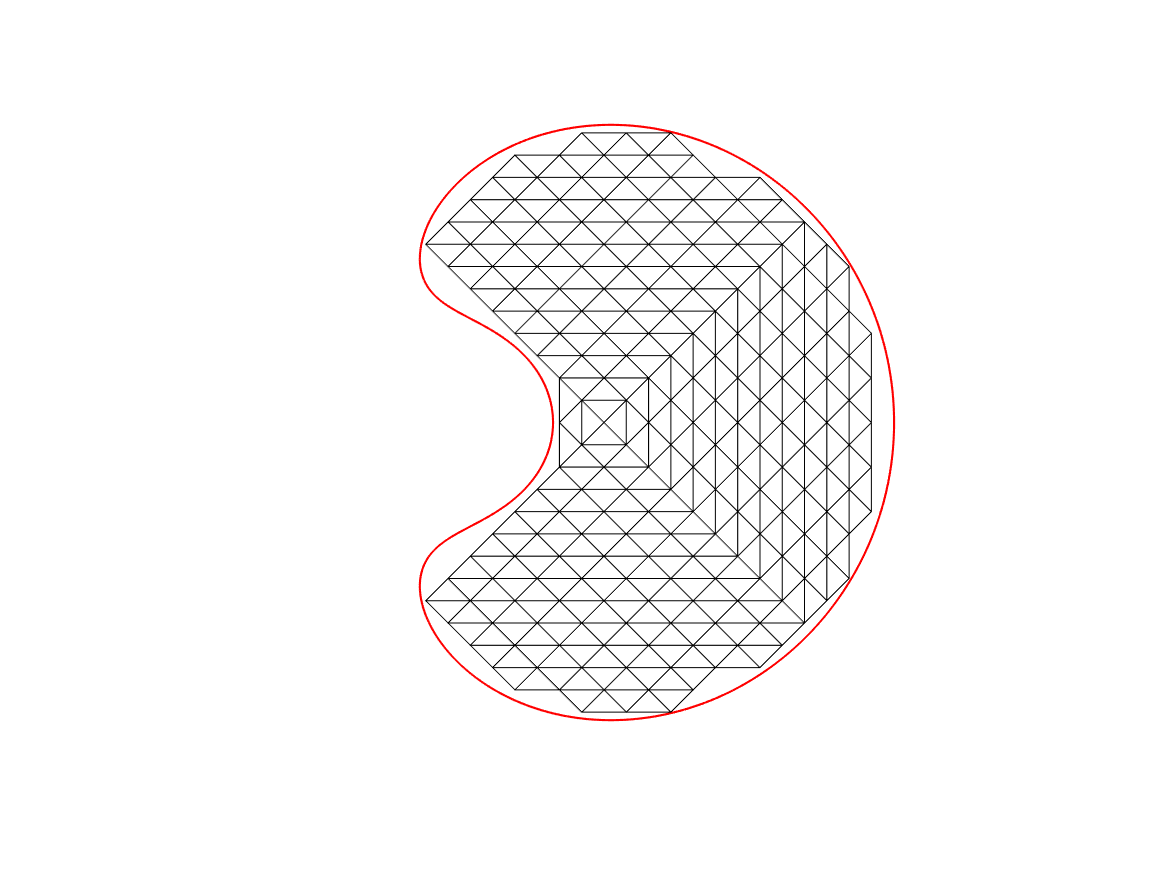}
  \captionof{figure}{Representation of a kidney-shaped domain.}
  \label{fig:kidney}
\end{minipage}
\end{figure}

\begin{table}[ht!]\renewcommand{\arraystretch}{1.1}\addtolength{\tabcolsep}{-4pt}
\centering
\begin{tabular}{c|c||c c||c c||c c}
$k$ & $N$ & $e_y$ & $\text{order}$ & $e_{\bm{q}}$ & $\text{order}$ & $e_{\widehat{y}}$ & $\text{order}$  \\\hline\hline
$0$ & $16$ & $8.26E-02$ & $-$ & $7.06E-01$ & $-$ & $1.16E-01$ & $-$ \\ 
  & $96$ & $7.04E-02$ & $0.18$ & $2.24E-01$ & $1.28$ & $3.97E-02$ & $1.60$ \\ 
  & $400$ & $3.60E-02$ & $0.94$ & $1.13E-01$ & $0.96$ & $1.95E-02$ & $1.00$ \\ 
  & $1680$ & $1.89E-02$ & $0.90$ & $5.79E-02$ & $0.93$ & $9.10E-03$ & $1.06$ \\ 
  & $7000$ & $1.02E-02$ & $0.86$ & $2.84E-02$ & $1.00$ & $3.86E-03$ & $1.20$ \\ 
  & $28504$ & $5.30E-03$ & $0.94$ & $1.41E-02$ & $0.99$ & $1.78E-03$ & $1.10$ \\ 
\hline \hline 
$1$ & $16$ & $2.56E-02$ & $-$ & $9.83E-02$ & $-$ & $1.18E-02$ & $-$ \\ 
  & $96$ & $5.65E-03$ & $1.68$ & $2.41E-02$ & $1.57$ & $1.05E-03$ & $2.70$ \\ 
  & $400$ & $1.50E-03$ & $1.86$ & $6.15E-03$ & $1.91$ & $2.05E-04$ & $2.28$ \\ 
  & $1680$ & $3.88E-04$ & $1.88$ & $1.47E-03$ & $1.99$ & $5.05E-05$ & $1.95$ \\ 
  & $7000$ & $9.44E-05$ & $1.98$ & $3.31E-04$ & $2.09$ & $5.10E-06$ & $3.21$ \\ 
  & $28504$ & $2.34E-06$ & $1.99$ & $7.95E-05$ & $2.03$ & $5.66E-07$ & $3.13$ \\ 
\hline \hline 
$2$ & $16$ & $9.34E-03$ & $-$ & $4.10E-02$ & $-$ & $9.38E-03$ & $-$ \\ 
  & $96$ & $4.62E-04$ & $3.36$ & $1.68E-03$ & $3.56$ & $2.14E-04$ & $4.22$ \\ 
  & $400$ & $6.48E-05$ & $2.76$ & $2.48E-04$ & $2.68$ & $3.71E-05$ & $2.45$ \\ 
  & $1680$ & $9.15E-06$ & $2.72$ & $5.06E-05$ & $2.21$ & $5.99E-06$ & $2.54$ \\ 
  & $7000$ & $8.82E-07$ & $3.28$ & $4.09E-06$ & $3.53$ & $3.34E-07$ & $4.05$ \\ 
  & $28504$ & $9.99E-08$ & $3.10$ & $3.58E-07$ & $3.47$ & $1.75E-07$ & $4.20$ \\ 
\hline \hline 
$3$ & $16$ & $7.15E-04$ & $-$ & $3.15E-03$ & $-$ & $6.61E-04$ & $-$ \\ 
  & $96$ & $2.06E-05$ & $3.96$ & $1.76E-04$ & $3.22$ & $1.29E-05$ & $4.39$ \\ 
  & $400$ & $1.59E-06$ & $3.59$ & $1.36E-05$ & $3.56$ & $1.03E-06$ & $3.54$ \\ 
  & $1680$ & $1.19E-07$ & $3.66$ & $9.12E-07$ & $3.76$ & $8.39E-08$ & $3.49$ \\ 
  & $7000$ & $4.99E-09$ & $4.46$ & $3.23E-08$ & $4.69$ & $1.86E-09$ & $5.33$ \\ 
  & $28504$ & $2.73E-10$ & $4.12$ & $1.44E-09$ & $4.44$ & $4.76E-11$ & $5.22$ \\ 
\end{tabular}
\caption{History of convergence history of the error in $y$, $\bm q$ and $\widehat{y}$ for Example 1.}
\label{table:y,p,yhat}
\end{table}

\begin{table}[ht!]\renewcommand{\arraystretch}{1.1}\addtolength{\tabcolsep}{-4pt}
\centering
\begin{tabular}{c|c||c c||c c||c c}
$k$ & $N$ & $e_z$ & $\text{order}$ & $e_{\boldsymbol{p}}$ & $\text{order}$ & $e_{\widehat{\boldsymbol{z}}}$ & $\text{order}$  \\\hline\hline
$0$ & $16$ & $1.43E-01$ & $-$ & $1.10E-00$ & $-$ & $1.65E-01$ & $-$ \\ 
  & $96$ & $1.16E-01$ & $0.02$ & $4.37E-01$ & $1.03$ & $4.41E-02$ & $1.47$ \\ 
  & $400$ & $5.72E-02$ & $0.99$ & $2.23E-01$ & $0.94$ & $2.21E-02$ & $0.97$ \\ 
  & $1680$ & $2.94E-02$ & $0.93$ & $1.11E-01$ & $0.97$ & $1.03E-02$ & $1.05$ \\ 
  & $7000$ & $1.62E-02$ & $0.84$ & $5.48E-02$ & $0.99$ & $4.31E-03$ & $1.22$ \\ 
  & $28504$ & $8.45E-03$ & $0.93$ & $2.73E-02$ & $0.98$ & $1.99E-03$ & $1.10$ \\ 
\hline \hline 
$1$ & $16$ & $4.67E-02$ & $-$ & $1.89E-01$ & $-$ & $2.18E-02$ & $-$ \\ 
  & $96$ & $1.30E-02$ & $1.42$ & $5.01E-02$ & $1.48$ & $3.56E-03$ & $2.04$ \\ 
  & $400$ & $3.36E-03$ & $1.90$ & $1.21E-02$ & $1.99$ & $7.72E-04$ & $2.13$ \\ 
  & $1680$ & $8.15E-04$ & $1.98$ & $2.63E-03$ & $2.13$ & $1.04E-04$ & $2.77$ \\ 
  & $7000$ & $1.98E-04$ & $1.99$ & $5.98E-04$ & $2.08$ & $9.71E-06$ & $3.33$ \\ 
  & $28504$ & $4.90E-05$ & $1.99$ & $1.45E-04$ & $2.02$ & $1.03E-06$ & $3.20$ \\ 
\hline \hline 
$2$ & $16$ & $7.26E-03$ & $-$ & $3.40E-02$ & $-$ & $4.52E-03$ & $-$ \\ 
  & $96$ & $8.27E-04$ & $2.42$ & $3.54E-03$ & $2.52$ & $2.89E-04$ & $3.07$ \\ 
  & $400$ & $1.40E-04$ & $2.49$ & $7.32E-04$ & $2.21$ & $8.66E-05$ & $1.69$ \\ 
  & $1680$ & $2.27E-05$ & $2.54$ & $1.22E-04$ & $2.49$ & $1.66E-05$ & $2.30$ \\ 
  & $7000$ & $1.93E-06$ & $3.45$ & $9.18E-06$ & $3.63$ & $8.79E-07$ & $4.12$ \\ 
  & $28504$ & $2.04E-07$ & $3.20$ & $8.07E-07$ & $3.46$ & $4.37E-08$ & $4.27$ \\ 
\hline \hline 
$3$ & $16$ & $9.36E-04$ & $-$ & $4.74E-03$ & $-$ & $6.50E-04$ & $-$ \\ 
  & $96$ & $8.27E-05$ & $2.71$ & $7.13E-04$ & $2.11$ & $6.83E-05$ & $2.51$ \\ 
  & $400$ & $6.41E-06$ & $3.59$ & $4.99E-05$ & $3.73$ & $5.22E-06$ & $3.60$ \\ 
  & $1680$ & $2.98E-07$ & $4.28$ & $1.89E-06$ & $4.56$ & $2.02E-07$ & $4.53$ \\ 
  & $7000$ & $1.21E-08$ & $4.49$ & $6.24E-08$ & $4.78$ & $3.51E-09$ & $5.68$ \\ 
  & $28504$ & $6.98E-10$ & $4.07$ & $2.89E-09$ & $4.39$ & $7.93E-11$ & $5.40$ \\ 
\end{tabular}
\caption{History of convergence history of the error in $z$, $\bm p$ and $\widehat{z}$ for Example 1.}
\label{table:z,r,zhat}
\end{table}

\begin{table}[ht!]\renewcommand{\arraystretch}{1.1}\addtolength{\tabcolsep}{-4pt}
\centering
\begin{tabular}{c|c||c c||c c||c c}
$k$ & $N$ & $e_y$ & $\text{order}$ & $e_{\boldsymbol{q}}$ & $\text{order}$ & $e_{\widehat{\boldsymbol{y}}}$ & $\text{order}$  \\\hline\hline
$0$ & $22$ & $1.85E-01$ & $-$ & $4.05E-01$ & $-$ & $8.30E-02$ & $-$ \\ 
  & $85$ & $9.17E-02$ & $1.04$ & $2.19E-01$ & $0.91$ & $4.73E-02$ & $0.83$ \\ 
  & $370$ & $5.10E-02$ & $0.80$ & $1.03E-01$ & $1.02$ & $1.78E-02$ & $1.32$ \\ 
  & $1576$ & $2.73E-02$ & $0.86$ & $5.01E-02$ & $0.99$ & $7.70E-03$ & $1.16$ \\ 
  & $6545$ & $1.43E-02$ & $0.91$ & $2.43E-02$ & $1.01$ & $3.35E-03$ & $1.17$ \\ 
  & $26590$ & $7.29E-03$ & $0.96$ & $1.21E-02$ & $1.00$ & $1.61E-03$ & $1.04$ \\ 
\hline \hline 
$1$ & $22$ & $3.31E-02$ & $-$ & $1.05E-01$ & $-$ & $9.66E-03$ & $-$ \\ 
  & $85$ & $9.22E-03$ & $1.89$ & $2.99E-02$ & $1.86$ & $2.67E-03$ & $1.90$ \\ 
  & $370$ & $2.24E-03$ & $1.93$ & $7.13E-03$ & $1.95$ & $4.84E-04$ & $2.32$ \\ 
  & $1576$ & $5.42E-04$ & $1.96$ & $1.44E-03$ & $2.21$ & $4.70E-05$ & $3.22$ \\ 
  & $6545$ & $1.34E-04$ & $1.96$ & $3.32E-04$ & $2.06$ & $5.77E-06$ & $2.95$ \\ 
  & $26590$ & $3.34E-05$ & $1.98$ & $7.98E-05$ & $2.03$ & $6.29E-07$ & $3.16$ \\ 
\hline \hline 
$2$ & $22$ & $3.74E-03$ & $-$ & $1.00E-02$ & $-$ & $9.31E-04$ & $-$ \\ 
  & $85$ & $6.64E-04$ & $2.56$ & $2.46E-03$ & $2.07$ & $4.58E-04$ & $1.05$ \\ 
  & $370$ & $7.63E-05$ & $2.94$ & $3.05E-04$ & $2.84$ & $4.16E-05$ & $3.26$ \\ 
  & $1576$ & $8.44E-06$ & $3.04$ & $3.72E-05$ & $2.90$ & $3.20E-06$ & $3.54$ \\ 
  & $6545$ & $9.63E-07$ & $3.05$ & $3.48E-06$ & $3.33$ & $2.10E-07$ & $3.83$ \\ 
  & $26590$ & $1.17E-07$ & $3.01$ & $3.81E-07$ & $3.16$ & $1.57E-08$ & $3.70$ \\ 
\hline \hline 
$3$ & $22$ & $4.72E-04$ & $-$ & $1.91E-03$ & $-$ & $2.66E-04$ & $-$ \\ 
  & $85$ & $5.56E-05$ & $3.16$ & $2.91E-04$ & $2.79$ & $4.79E-05$ & $2.53$ \\ 
  & $370$ & $2.95E-06$ & $4.00$ & $2.03E-05$ & $3.62$ & $2.34E-06$ & $4.11$ \\ 
  & $1576$ & $1.12E-07$ & $4.51$ & $7.13E-07$ & $4.62$ & $5.42E-08$ & $5.20$ \\ 
  & $6545$ & $6.04E-09$ & $4.10$ & $3.05E-08$ & $4.43$ & $1.64E-09$ & $4.99$ \\ 
  & $26590$ & $3.56E-10$ & $4.04$ & $1.49E-09$ & $4.30$ & $4.38E-11$ & $5.17$ \\ 
\end{tabular}
\caption{History of convergence history of the error in $y$, $\bm q$ and $\widehat{y}$ for the example 2.}
\label{table:y,p,yhat_kidney}
\end{table}

\begin{table}[ht!]\renewcommand{\arraystretch}{1.1}\addtolength{\tabcolsep}{-4pt}
\centering
\begin{tabular}{c|c||c c||c c||c c}
$k$ & $N$ & $e_z$ & $\text{order}$ & $e_{\boldsymbol{p}}$ & $\text{order}$ & $e_{\widehat{\boldsymbol{z}}}$ & $\text{order}$  \\\hline\hline
$0$ & $22$ & $3.18E-01$ & $-$ & $8.49E-01$ & $-$ & $8.95E-02$ & $-$ \\ 
  & $85$ & $1.53E-01$ & $1.08$ & $4.27E-01$ & $1.01$ & $5.33E-02$ & $0.77$ \\ 
  & $370$ & $8.56E-02$ & $0.78$ & $2.06E-01$ & $0.99$ & $1.93E-02$ & $1.38$ \\ 
  & $1576$ & $4.50E-02$ & $0.86$ & $1.03E-01$ & $0.97$ & $8.45E-03$ & $1.14$ \\ 
  & $6545$ & $2.36E-02$ & $0.91$ & $5.11E-02$ & $0.98$ & $3.65E-03$ & $1.18$ \\ 
  & $26590$ & $1.20E-02$ & $0.96$ & $2.56E-02$ & $0.99$ & $1.74E-03$ & $1.06$ \\ 
\hline \hline 
$1$ & $22$ & $6.54E-02$ & $-$ & $1.72E-01$ & $-$ & $1.95E-02$ & $-$ \\ 
  & $85$ & $1.71E-02$ & $1.99$ & $4.37E-02$ & $2.03$ & $3.81E-03$ & $2.42$ \\ 
  & $370$ & $4.16E-03$ & $1.92$ & $9.45E-03$ & $2.08$ & $5.07E-04$ & $2.74$ \\ 
  & $1576$ & $1.01E-03$ & $1.95$ & $2.13E-03$ & $2.06$ & $3.97E-05$ & $3.52$ \\ 
  & $6545$ & $2.51E-04$ & $1.96$ & $5.11E-04$ & $2.00$ & $5.04E-06$ & $2.90$ \\ 
  & $26590$ & $6.26E-05$ & $1.98$ & $1.26E-04$ & $2.00$ & $5.75E-07$ & $3.01$ \\ 
\hline \hline 
$2$ & $22$ & $8.15E-03$ & $-$ & $2.49E-02$ & $-$ & $4.05E-03$ & $-$ \\ 
  & $85$ & $1.74E-03$ & $2.29$ & $7.17E-03$ & $1.83$ & $1.39E-03$ & $1.58$ \\ 
  & $370$ & $1.94E-04$ & $2.98$ & $9.31E-04$ & $2.78$ & $1.38E-04$ & $3.14$ \\ 
  & $1576$ & $1.98E-05$ & $3.15$ & $9.14E-05$ & $3.20$ & $1.07E-05$ & $3.54$ \\ 
  & $6545$ & $1.99E-06$ & $3.22$ & $7.81E-06$ & $3.45$ & $5.75E-07$ & $4.10$ \\ 
  & $26590$ & $2.33E-07$ & $3.06$ & $8.17E-07$ & $3.22$ & $3.77E-08$ & $3.89$ \\ 
\hline \hline 
$3$ & $22$ & $1.30E-03$ & $-$ & $4.42E-03$ & $-$ & $8.10E-04$ & $-$ \\ 
  & $85$ & $1.15E-04$ & $3.59$ & $5.08E-04$ & $3.20$ & $9.06E-05$ & $3.24$ \\ 
  & $370$ & $5.30E-06$ & $4.18$ & $2.78E-05$ & $3.95$ & $3.53E-06$ & $4.41$ \\ 
  & $1576$ & $2.21E-07$ & $4.38$ & $1.23E-06$ & $4.30$ & $6.00E-08$ & $5.62$ \\ 
  & $6545$ & $1.33E-08$ & $3.95$ & $6.71E-08$ & $4.09$ & $2.38E-09$ & $4.53$ \\ 
  & $26590$ & $8.15E-10$ & $3.98$ & $3.38E-09$ & $4.26$ & $9.70E-11$ & $4.57$ \\ 
\end{tabular}
\caption{History of convergence history of the error in $z$, $\bm p$ and $\widehat{z}$ for the example 2.}
\label{table:z,r,zhat_kidney}
\end{table}

\section{Conclusions}\label{ch:conclusions}

We have introduced and analyzed an unfitted high-order HDG method for a convection-diffusion control problem by applying the Transfer Path Method (TPM). The method was analyzed by considering the direction of $\boldsymbol{m}$ for the construction of the transferring paths, thus filling a gap in the existing literature. The method is well-posed and exhibits theoretically optimal convergence rates in the $L^2$-norm for all unknowns when the distance between $\partial \Omega$ and the computational boundary $\partial \Omega_h$ is proportional to $h^{\delta + 1}$, with $\delta > 0$. We highlight that the numerical experiments demonstrate optimal convergence even in the limiting case where $\delta = 0$.

\subsubsection*{Acknowledgements}
Manuel Solano was partially supported by ANID-Chile through Basal Project FB210005 and Fondecyt 1240183. MS was also supported by the Swedish Research Council under grant no. 2021-06594 while in residence at Institut Mittag-Leffler in Djursholm, Sweden during the fall semester of 2025.

\appendix

\section{Proofs of previous results}
\subsection{Proof of Lemma \ref{lem:Lambdas}}
\label{sec:proof-lambdas}

Let  $\bm{p} \in [H^1(K^e_\mathrm{ext})]^d$. Following the steps of \cite[Lemma 5.2]{paper2} (integration by parts), we have that 
\begin{equation*}
    l(\boldsymbol{x}) \Lambda^{\boldsymbol{p}}(\boldsymbol{x})
    = \int_0^{l(\boldsymbol{x})} \partial_t (\boldsymbol{p}(\boldsymbol{x} + t \boldsymbol{m}(\boldsymbol{x}))\cdot \boldsymbol{m}(\boldsymbol{x}))(l(\boldsymbol{x}) - t) \text{d}t.
\end{equation*}
Using Cauchy-Schwarz inequality and defining $\boldsymbol{w}(\boldsymbol{x},t) := \partial_t (\boldsymbol{p}(\boldsymbol{x} + t \boldsymbol{m}(\boldsymbol{x})) \cdot \boldsymbol{m}(\boldsymbol{x}))$, we deduce that
\begin{equation}
    \label{eq:proof-lambdas-aux1}
    l^2(\boldsymbol{x})(\Lambda^{\boldsymbol{p}}(\boldsymbol{x}))^2
    \leq \dfrac{l^3(\boldsymbol{x})}{3} \int_0^{l(\boldsymbol{x})} \boldsymbol{w}^2(\boldsymbol{x}, t) \text{d} t
    \leq \dfrac{l(\boldsymbol{x})(H_e^{\perp})^2}{3} \tnorm{\boldsymbol{w}}_e^2.
\end{equation}
where for the last inequality we used that $H_e^{\perp} = r_e h_e$. Since $\Gamma$ is Lipschitz and piecewise $\mathcal{C}^2$, by \cite[Lemma 3.4]{OyZuSo2019} we have
\begin{equation}
    \label{eq:proof-lambdas-aux2}
    \norm[0]{\Lambda^{\boldsymbol{p}}}_{e,l} 
    \leq r_e \beta^{-1/2}_e \dfrac{1}{\sqrt{3}} \norm[0]{\boldsymbol{w}}_{K_{ext}^e, h^2}.
\end{equation}
On the other hand, by the chain rule
\begin{equation*}
    \boldsymbol{w}(\boldsymbol{x}, t)
    = \partial_t (\boldsymbol{p}(\boldsymbol{\xi}) \cdot \boldsymbol{m}(\boldsymbol{x})) 
    = (\nabla\boldsymbol{p} \,\boldsymbol{m}(\boldsymbol{x})) \cdot \boldsymbol{m}(\boldsymbol{x}).
\end{equation*}
Then, considering that $|\boldsymbol{m}(\boldsymbol{x})| \leq 1$, and combining the above equality with \eqref{eq:proof-lambdas-aux2}, we deduce that \eqref{ineq:Lambdas-a} holds. On the other hand, from \eqref{eq:proof-lambdas-aux1}, considering now $\boldsymbol{p} \in [\mathbb{P}_k(K_e)]^d$ and using the definition of $C_{ext}^e$, we deduce
\begin{equation}
    \norm[0]{l^{1/2} \Lambda^{\boldsymbol{p}}}_e
    \leq \dfrac{r_e^{3/2} h_e}{3} C_{ext}^e \norm[0]{\nabla \boldsymbol{p}}_{K^e}
    \leq r_{e}^{3/2} C_{ext}^e C_{inv} \norm[0]{\boldsymbol{p}}_{K^e},
\end{equation}
where the last inequality was obtained after applying an inverse inequality (see \cite[Lemma 1.44]{di2011mathematical}) and $C_{inv} > 0$ is a constant independent of $h$.


\bibliographystyle{abbrvnat}
\bibliography{refs}


\end{document}